\documentclass[pdflatex,sn-mathphys-num]{sn-jnl}

\usepackage{amsmath,amssymb,amsfonts}%
\usepackage{amsthm}%
\usepackage{accents}%
\usepackage{algorithm}%
\usepackage{algorithmicx}%
\usepackage{algpseudocode}%
\usepackage[title]{appendix}%
\usepackage{booktabs}%
\usepackage{comment}
\usepackage{graphicx}%
\usepackage{listings}%
\usepackage{manyfoot}%
\usepackage{mathrsfs}%
\usepackage{mathtools}%
\usepackage{multirow}
\usepackage{textcomp}%
\usepackage{tikz-cd}%
\usepackage{xcolor}%

\makeatletter
\def\bar{\accentset{{\cc@style\underline{\mskip13mu}}}}
\makeatother

\theoremstyle{thmstyleone}%
\newtheorem{theorem}{Theorem}[section]
\newtheorem{definition}[theorem]{Definition}%
\newtheorem{lemma}[theorem]{Lemma}%
\newtheorem{proposition}[theorem]{Proposition}%
\newtheorem{corollary}[theorem]{Corollary}%

\theoremstyle{thmstyletwo}%
\newtheorem{remark}[theorem]{Remark}%

\theoremstyle{thmstylethree}%
\numberwithin{equation}{section}

\begin{document}
\title{On superspecial hyperelliptic curves of genus 5 whose automorphism groups contain $(\mathbb{Z}/2\mathbb{Z})^3$}


\author*[1]{\fnm{Ryo} \sur{Ohashi}}\email{ryo-ohashi@g.ecc.u-tokyo.ac.jp}
\author[2]{\fnm{Momonari} \sur{Kudo}}\email{m-kudo@fit.ac.jp}

\affil*[1]{\orgdiv{Graduate School of Information Science and Technology}, \orgname{The University of Tokyo}, \orgaddress{\street{7-3-1 Hongo}, \city{Bunkyo-ku}, \postcode{113-0033}, \state{Tokyo}, \country{Japan}}}
\affil[2]{\orgname{Fukuoka Institute of Technology}, \orgaddress{\street{3-30-1 Wajiro-higashi}, \city{Higashi-ku}, \postcode{811-0295}, \state{Fukuoka}, \country{Japan}}}


\abstract{While the numbers of superspecial curves of genus at most 3 are well understood, and several computational approaches have been developed to count superspecial curves of genus 4 with large automorphism groups, much less is known in higher genera.
In this paper, we construct a feasible algorithm to enumerate superspecial hyperelliptic curves of genus 5 whose automorphism groups contain $\hspace{-0.3mm}(\mathbb{Z}/2\mathbb{Z})^3\hspace{-0.2mm}$.
We implement and executing our algorithm in Magma, we succeeded in enumerating such superspecial curves in every characteristic $\hspace{-0.2mm}11 < p < 1000$.}




\maketitle

\section{ Introduction}\label{sec1}
Throughout, all the complexities are measured by the number of arithmetic operations in $\mathbb{F}_{p^2}\hspace{-0.3mm}$ for a prime $p$, unless otherwise noted.
Soft-O notation omits logarithmic factors.
A \emph{curve} means a non-singular projective variety of dimension one.
Let $K$ be a field of characteristic $p > 0$, and $\bar{K}$ its algebraic closure.
A genus-$g$ curve $C$ defined over $K$ is said to be \emph{superspecial} when its Jacobian variety is isomorphic (over $\bar{K}$) to a product of supersingular elliptic curves.
Superspecial curves are important objects not only in theory but also in practical applications such as isogeny-based cryptography, which is a candidate for post-quantum cryptography.

For a given pair $(g,p)$, it is known that there exist only finitely many superspecial genus-$g$ curves over $\bar{\mathbb{F}}_p$, and thus the problem of determining the number of them is of fundamental importance.
For the cases where $g \leq 3$, this problem has been completely solved for all $p > 0$, based on the theory of principally polarized abelian varieties.
We refer to \cite{Kudo} for a survey of related works.

In contrast to the cases where $g \leq 3$, the problem for $g \geq 4$ has not been solved for all primes.
However, over the last decade, several computational approaches for $g=4$ have been proposed; \,see \cite{KH17},\,\cite{KHH20} for the non-hyperelliptic cases, and \cite{KH22},\,\cite{OKH23},\,\cite{KNT23},\,\cite{OK24},\,\cite{TOT25} for the hyperelliptic cases.
In particular, Kudo, Nakagawa, and Takagi proposed in \cite{KNT23} an algorithm to find a superspecial hyperelliptic genus-4 curve with automorphism group containing $\mathbb{Z}/6\mathbb{Z}$, whose extended abstract was presented at CASC2022.
In \cite{OKH23} and \cite{OK24}, the authors presented enumeration algorithms for superspecial hyperelliptic curves of genus 4 whose automorphism groups contain the Klein 4-group (or the dihedral group of order 8).
Their computational results in \cite{OK24} implies an expectation that there exists a superspecial hyperelliptic curve of genus 4 in characteristic $p$ for any $p$ with $p \geq 19$.

The next target is the case where $g=5$; however, there is almost no prior work on the enumeration of superspecial genus-5 curves (except for Kudo-Harashita's result~\cite{KH18} on the trigonal case in characteristic $p \leq 13$).
In this paper, analogously to our prior works~\cite{OKH23} and \cite{OK24} for genus-4 curves, we restrict our attention to hyperelliptic curves of genus 5 with automorphism group containing $(\mathbb{Z}/2\mathbb{Z})^3$.
The main reason for studying such curves is that, via the decomposition of Jacobian varieties described in Section~\ref{sec3}, their superspeciality can be reduced to that of curves of genus at most 2, which allows us to construct superspecial curves efficiently.
Exploiting this property, we construct a feasible algorithm to enumerate such curves in characteristic $p > 11$; \,for an explicit description, see Section~\ref{sec5}.
Also, we can show that all the computations in our algorithm are carried out over $\mathbb{F}_{p^2}$, and the complexity of our algorithm is proved to be $\widetilde{O}(p^3)$.
By implementing and executing the algorithm in \textsf{Magma} Computational Algebra System, we obtain the following main result:\vspace{-4.5mm}
\begin{theorem}\label{thm:main}
There exists a superspecial hyperelliptic curve of genus 5 whose automorphism group contains $(\mathbb{Z}/2\mathbb{Z})^3$ in any characteristic $23 \leq p < 1000$, except for
\begin{align*}
    p =\, & 37,41,43,53,61,67,73,97,181,193,197,233,241,277,313,331,373,421,\\[-1.1mm]
    & 541,571,613,643,673,709,769,877,977.
\end{align*}
The number of $\bar{\mathbb{F}}_p$-isomorphism
classes of such curves are summarized in Tables \ref{tbl:13-450} and \ref{tbl:450-1000}.
\end{theorem}

The rest of this paper is structured as follows:
\,Section~\ref{sec2} reviews the classification of hyperelliptic genus-5 curves according to their automorphism groups.
In Sections~\ref{sec3} and \ref{sec4}, we prove several properties of hyperelliptic genus-5 curves whose automorphism groups contain $(\mathbb{Z}/2\mathbb{Z})^3$.
In Section~\ref{sec5}, we propose an algorithm for enumerating such curves that are superspecial, and we present the experimental results of executing our algorithm in Section~\ref{sec6}.
Finally, Section~\ref{sec7} concludes the paper with some remarks.

\paragraph*{Acknowledgments.}
This work was supported by JSPS Grant-in-Aid for Young
Scientists JP25K17225 and JP23K12949.

\section{ Classification of hyperelliptic genus-5 curves}\label{sec2}
In the following, we use the following notation for the finite groups:\vspace{-1.9mm}
\begin{alignat*}{3}
    &\mathbf{C}_n: \,\mbox{the cyclic group of order $n$}; && \mathbf{D}_{2n}: \,\mbox{the dihedral group of order $2n$};\\[-1mm]
    & \mathbf{S}_n: \,\mbox{the symmetric group of degree $n$}; \quad && \mathbf{A}_n: \,\mbox{the alternating group of degree $n$}.\\[-5.6mm]
\end{alignat*}
For a hyperelliptic curve $H$ defined over an algebraically closed field $k$, we denote the reduced (resp.\ full) automorphism group of $H$ over $k$ by $\overline{\mathrm{Aut}}(H)$ (resp.\ ${\rm Aut}(H)$).
The possible automorphism groups of hyperelliptic curves of genus $g$ in characteristic 0 are classified by \cite[Table 1]{Shaska}.
His result can be extended to an algebraically closed field $k$ of characteristic $p > g+1$ with $p \neq 2g+1$ following the discussion in \cite[Section 3.1]{LR}.\\
By writing down the case $g=5$ explicitly, we obtain the following result:\vspace{-4.5mm}
\begin{proposition}\label{prp:classification}
Let $H$ be a hyperelliptic genus-5 curve over an algebraically closed field $k$ of characteristic $p \geq 0$ with $p \neq 2,3,5,11$.
The reduced (resp. full) automorphism group of $H$ is one of those listed in the second (resp. third) column of Table~\ref{table:classification}.
The hyperelliptic curve $H$ for each type is isomorphic to $y^2 = f(x)$ in the fourth column of Table~\ref{table:classification}.
\begin{table}[htbp]
    \centering
    \caption{Possible automorphism groups of genus-5 hyperelliptic curves $H$ defined over an algebraically closed field $k$ of characteristic $p \geq 0$ with $p \neq 2,3,5,11$, together with defining equations $y^2=f(x)$ for $H$, where $a,b,c,d,e \in k$.}\label{table:classification}\vspace{-2.5mm}
    \begin{tabular}{c||c|c|l}
        Type & $\overline{\mathrm{Aut}}(H)$ & ${\mathrm{Aut}}(H)$ & Defining equation $y^2 = f(x)$ for $H$ (in normal form)\\\hline
        {\bf 1} & $\{0\}$ & $\mathbf{C}_2$ & $y^2 = (\text{square-free polynomial in $x$ of degree $11$ or $12$})$\\\hline
        {\bf 2-1} & \multirow{2}{*}{$\mathbf{C}_2$} & $\mathbf{C}_2^2$ & $y^2 = (x^2-1)(x^2-a)(x^2-b)(x^2-c)(x^2-d)(x^2-e)$\\
        {\bf 2-2} & & $\mathbf{C}_4$ & $y^2 = x(x^2-1)(x^2-a)(x^2-b)(x^2-c)(x^2-d)$\\\hline
        {\bf 3} & $\mathbf{C}_3$ & $\mathbf{C}_6$ & $y^2 = (x^3-1)(x^3-a)(x^3-b)(x^3-c)$\\\hline
        {\bf 4-1} & \multirow{2}{*}{$\mathbf{C}_2^2$} & $\mathbf{C}_2^3$ & $y^2 = (x^4+ax^2+1)(x^4+bx^2+1)(x^4+cx^2+1)$\\[-0.3mm]
        {\bf 4-2} & & $\mathbf{C}_2 \times \mathbf{C}_4$ & $y^2 = (x^4-1)(x^4+ax^2+1)(x^4+bx^2+1)$\\\hline
        {\bf 5} & $\mathbf{C}_4$ & $\mathbf{C}_2 \times \mathbf{C}_4$ & $y^2 = (x^4-1)(x^4-a)(x^4-b)$\\\hline
        {\bf 6-1} & \multirow{2}{*}{$\mathbf{D}_6$} & $\mathbf{D}_{12}$ & $y^2 = (x^6+ax^3+1)(x^6+bx^3+1)$\\[-0.3mm]
        {\bf 6-2} & & $\mathbf{C}_3 \rtimes \mathbf{C}_4$ & $y^2 = (x^6-1)(x^6+ax^3+1)$\\\hline
        {\bf 7} & $\mathbf{D}_8$ & $\mathbf{C}_2^2 \rtimes \mathbf{C}_4$ & $y^2 = (x^4-1)(x^8+ax^4+1)$\\\hline
        {\bf 8} & $\mathbf{D}_{10}$ & $\mathbf{D}_{20}$ & $y^2 = x^{11}+ax^6+x$\\\hline
        {\bf 9} & $\mathbf{D}_{12}$ & $\mathbf{C}_2 \times \mathbf{D}_{12}$ & $y^2 = x^{12}+ax^6+1$\\\hline
        {\bf 10} & $\mathbf{A}_4$ & $\mathbf{C}_2 \times \mathbf{A}_4$ & $y^2 = x^{12}-ax^{10}-33x^8+2ax^6-33x^4-ax^2+1$\\\hline
        {\bf 11} & $\mathbf{S}_4$ & $\mathbf{A}_4 \rtimes \mathbf{C}_4$ & $y^2 = x^{12}-33x^8-33x^4+1$\\\hline
        {\bf 12} & $\mathbf{A}_5$ & $\mathbf{C}_2 \times \mathbf{A}_5$ & $y^2 = x^{11}+11x^6-x$\\\hline
        {\bf 13} & $\mathbf{C}_{11}$ & $\mathbf{C}_{22}$ & $y^2 = x^{11}+1$\\\hline
        {\bf 14} & $\mathbf{D}_{20}$ & $\mathbf{C}_4 \times \mathbf{D}_{10}$ & $y^2 = x^{11}+x$\\\hline
        {\bf 15} & $\mathbf{D}_{24}$ & $\mathbf{D}_{12} \rtimes \mathbf{C}_4$ & $y^2 = x^{12}+1$
    \end{tabular}
\end{table}
\end{proposition}\vspace{-10.5mm}

In particular, the superspeciality of the last three curves (i.e., Types {\bf 13},\,{\bf 14},\,{\bf 15}) in Table \ref{table:classification} can be determined immediately as follows:\vspace{-4.6mm}
\begin{proposition}
Assume that $p \geq 11$.
Then, the following statements are true:\vspace{-0.7mm}
\begin{enumerate}
\item[{\rm (a)}] The genus-5 curve $y^2 = x^{11} + 1$ is superspecial if and only if $p \equiv 11 \pmod{10}$.
\item[{\rm (b)}] The genus-5 curve $y^2 = x^{11} + x$ is superspecial if and only if $p \equiv 11,19 \pmod{20}$.
\item[{\rm (c)}] The genus-5 curve $y^2 = x^{12} + 1$ is superspecial if and only if $p \equiv 11 \pmod{12}$.\vspace{-2mm}
\end{enumerate}
\begin{proof}
The assertions (a) and (b) are immediate consequences of \cite[Theorem 2]{Valentini} with $g=5$, and the assertion (c) follows from \cite[Example 2.9]{Ohashi-RIMS}.
\end{proof}
\end{proposition}

In this paper, we will investigate the superspeciality of hyperelliptic genus-5 curves defined by the equation\vspace{-1.8mm}
\begin{equation}\label{eq:Habc}
    H_{a,b,c}: y^2 = (x^4+ax^2+1)(x^4+bx^2+1)(x^4+cx^2+1)
\end{equation}
with $a,b,c \notin \{\pm 2\}$ and $a \neq b \neq c \neq a$.
Remark that these conditions on $a,b$, and $c$ are necessary and sufficient for $H_{a,b,c}$ to be non-singular.
Then, one can easily check that the curve $H_{a,b,c}$ has three order-2 automorphisms\vspace{-2.2mm}
\begin{align*}
    \begin{split}
    \iota: (x,y) &\longmapsto (x,-y),\\[-0.3mm]
    \sigma: (x,y) &\longmapsto (-x,y),\\[-0.6mm]
    \tau: (x,y) &\longmapsto (1/x,y/x^6),\\[-1.4mm]
    \end{split}
\end{align*}
where $\iota$ is the hyperelliptic involution, and therefore ${\mathrm{Aut}}(H_{a,b,c})$ contains a subgroup isomorphic to $\mathbf{C}_2^3$.
According to Proposition \ref{prp:classification}, such a curve is of one of the following seven types: Type {\bf 4-1},\,{\bf 7},\,{\bf 9},\,{\bf 10},\,{\bf 11},\,{\bf 12}, or {\bf 15}.
The diagram below illustrates inclusion relations among the automorphism groups of these types:
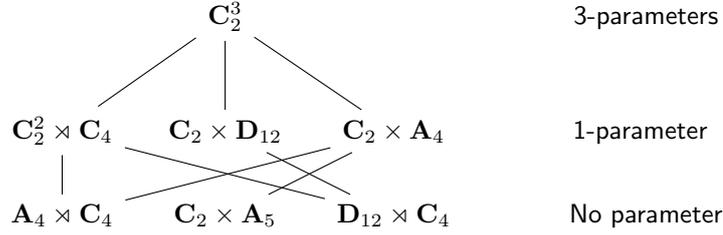
\vspace{-3.5mm}
\begin{figure}[htbp]
    \centering
    \begin{tikzcd}[column sep = small]
        & & \mathbf{C}_2^3 \arrow[dl,dash] \arrow[d,dash] \arrow[dr,dash] & & & & \textsf{3-parameters}\\[2.5mm]
        & \mathbf{C}_2^2 \rtimes \mathbf{C}_4 \arrow[d,dash] \arrow[drr,dash] & \mathbf{C}_2 \times \mathbf{D}_{12} \arrow[dr,dash] & \mathbf{C}_2 \times \mathbf{A}_4 \arrow[dll,dash] \arrow[dl,dash] & & & \hspace{-1.5mm}\textsf{1-parameter}\\[-1mm]
        & \mathbf{A}_4 \rtimes \mathbf{C}_4  & \mathbf{C}_2 \times \mathbf{A}_5 & \mathbf{D}_{12} \rtimes \mathbf{C}_4 & & & \textsf{No parameter}
    \end{tikzcd}
\end{figure}\vspace{-3.7mm}

\noindent The next section is devoted to proving general properties of curves of the form \eqref{eq:Habc}, while curves of more special types are treated individually in Section \ref{sec4}.

\section{ General properties of our curves}\label{sec3}
As in the previous section, let $k$ be an algebraically closed field of characteristic $p \geq 0$ where $p \notin \{2,3,5,11\}$.
In this section, we consider a hyperelliptic genus-5 curve $H_{a,b,c}$ as in equation \eqref{eq:Habc} for distinct $a,b,c \in k \!\smallsetminus\! \{\pm 2\}$.
We first provide a sufficient condition for another curve $H_{a',b',c'}$ to be isomorphic to $H_{a,b,c}$.\vspace{-4.5mm}
\begin{proposition}\label{prp:isomof4-1}
The two curves $H_{a,b,c}$ and $H_{a',b',c'}$ are isomorphic to each other 
if one of the following conditions holds:\vspace{-0.5mm}
\begin{enumerate}
\item[{\rm 1.}] $\{a',b',c'\} = \{a,b,c\}$,\vspace{0.2mm}
\item[{\rm 2.}] $\{a',b',c'\} = \{-a,-b,-c\}$,\vspace{0.2mm}
\item[{\rm 3.}] $\{a',b',c'\} = \{\frac{12-2a}{2+a},\frac{12-2b}{2+b},\frac{12-2c}{2+c}\}$,\vspace{-0.3mm}
\item[{\rm 4.}] $\{a',b',c'\} = \{-\frac{12-2a}{2+a},-\frac{12-2b}{2+b},-\frac{12-2c}{2+c}\}$,\vspace{-0.3mm}
\item[{\rm 5.}] $\{a',b',c'\} = \{\frac{12+2a}{2-a},\frac{12+2b}{2-b},\frac{12+2c}{2-c}\}$,\vspace{-0.3mm}
\item[{\rm 6.}] $\{a',b',c'\} = \{-\frac{12+2a}{2-a},-\frac{12+2b}{2-b},-\frac{12+2c}{2-c}\}$.\vspace{-2.7mm}
\end{enumerate}
\begin{proof}
We note that a permutation of $(a,b,c)$ clearly induces an isomorphism.
For each of the six cases, we explicitly construct an isomorphism $\psi: H_{a',b',c'} \rightarrow H_{a,b,c}$.
Here, we denote by $i$ an element of $k$ such that $i^2=-1$.
\begin{enumerate}
\item In this case, there exists an isomorphism $\psi: (x,y) \longmapsto (x,y)$.
\item In this case, there exists an isomorphism $\psi: (x,y) \longmapsto (ix,y)$.
\item In this case, there exists an isomorphism\vspace{-2mm}
\[
    \psi: (x,y) \longmapsto \biggl(\frac{x+1}{x-1},\frac{\kappa y}{(x-1)^6}\biggr) \,\text{ with }\, \kappa^2 = (2+a)(2+b)(2+c).\vspace{-1.8mm}
\]
\item In this case, there exists an isomorphism\vspace{-2mm}
\[
    \psi: (x,y) \longmapsto \biggl(\frac{x+i}{x-i},\frac{\kappa y}{(x-i)^6}\biggr) \,\text{ with }\, \kappa^2 = (2+a)(2+b)(2+c).\vspace{-1.8mm}
\]
\item In this case, there exists an isomorphism\vspace{-2mm}
\[
    \psi: (x,y) \longmapsto \biggl(i\frac{x+1}{x-1},\frac{\kappa y}{(x-1)^6}\biggr) \,\text{ with }\, \kappa^2 = (2-a)(2-b)(2-c).\vspace{-1.8mm}
\]
\item In this case, there exists an isomorphism\vspace{-2mm}
\[
    \psi: (x,y) \longmapsto \biggl(i\frac{x+i}{x-i},\frac{\kappa y}{(x-i)^6}\biggr) \,\text{ with }\, \kappa^2 = (2-a)(2-b)(2-c).\vspace{-1.3mm}
\]
\end{enumerate}
This completes the proof.
\end{proof}
\end{proposition}\vspace{-1.5mm}

By explicitly describing the decomposition of the Jacobian of $H_{a,b,c}$, we obtain the following result:\vspace{-4.7mm}
\begin{theorem}\label{thm:jacof4-1}
The Jacobian of $H_{a,b,c}$ admits a decomposition
into the product $E_1 \times E_2 \times E_3 \times {\rm Jac}(C)$, where $E_1,E_2,E_3$ are the elliptic curves defined by the equations\vspace{0.9mm}
\begin{align}\label{eq:ellof4-1}
    \begin{split}
    E_1: & \ Y^2 = (X-\lambda)(X-\mu)(x-\nu),\\[-1.2mm]
    E_2: & \ Y^2 = X(X-\lambda)(X-\mu)(x-\nu),\\[-1mm]
    E_3: & \ Y^2 = (X-1)(X-\lambda)(X-\mu)(x-\nu)\\[1.6mm]
    \end{split}
\end{align}
and $C$ is the genus-2 curve defined by the equation\vspace{1.3mm}
\[
    C: Y^2 = X(X-1)(X-\lambda)(X-\mu)(x-\nu),\vspace{1.2mm}
\]
with\vspace{-0.6mm}
\begin{equation}\label{eq:lmn}
    \lambda \coloneqq -\frac{2+a}{2-a}, \quad \mu \coloneqq -\frac{2+b}{2-b},\ \text{ and }\ \,\nu \coloneqq -\frac{2+c}{2-c}.\vspace{1.9mm}
\end{equation}
Note that $\lambda$, $\mu$, and $\nu$ are distinct from each other and from $0$ and $1$.\vspace{-1.4mm}
\begin{proof}
The defining equation of $H_{a,b,c}$ can be rewritten as\vspace{1.1mm}
\[
    H_{a,b,c}: y^2 = (x^2-\alpha^2)(x^2-1/\alpha^2)(x^2-\beta^2)(x^2-1/\beta^2)(x^2-\gamma^2)(x^2-1/\gamma^2)\vspace{2.3mm}
\]
where $\alpha,\beta,\gamma \in k$ with $\alpha^2+1/\alpha^2 = -a,\ \beta^2+1/\beta^2 = -b,\ \gamma^2+1/\gamma^2 = -c$.
Then, the quotient curve $H_1 \coloneqq H_{a,b,c}/\langle\sigma\rangle$ is given by\vspace{0.9mm}
\[
    H_1: v^2 = (u-\alpha^2)(u-1/\alpha^2)(u-\beta^2)(u-1/\beta^2)(u-\gamma^2)(u-1/\gamma^2)\vspace{2.6mm}
\]
via $u = x^2$ and $v = y$.
Also, the quotient curve $H_2 \coloneqq H_{a,b,c}/\langle\sigma \circ \iota\rangle$ is given by\vspace{1mm}
\[
    H_2: v^2 = u(u-\alpha^2)(u-1/\alpha^2)(u-\beta^2)(u-1/\beta^2)(u-\gamma^2)(u-1/\gamma^2)\vspace{2.6mm}
\]
via $u = x^2$ and $v = xy$.
Moreover, it follows from \cite[Theorem 1]{KT} that the Jacobian of $H_{a,b,c}$ is Richelot isogenous to the product of the Jacobian of $H_1$ and the Jacobian of $H_2$.
\newpage
Next, the Möbius transformation $u \mapsto\hspace{-0.1mm} \frac{u-1}{u+1} \eqqcolon X$ gives the correspondences
\begin{alignat*}{3}
    0 &\longmapsto -1, &\quad \infty &\longmapsto 1,\\[-0.9mm]
    \alpha^2 &\longmapsto \frac{\alpha^2-1}{\alpha^2+1} \eqqcolon \alpha', &\quad 1/\alpha^2 &\longmapsto \frac{1-\alpha^2}{1+\alpha^2} = -\alpha'\\[-0.7mm]
    \beta^2 &\longmapsto \frac{\beta^2-1}{\beta^2+1} \eqqcolon \beta', &\quad 1/\beta^2 &\longmapsto \frac{1-\beta^2}{1+\beta^2} = -\beta',\\[-0.7mm]
    \gamma^2 &\longmapsto \frac{\gamma^2-1}{\gamma^2+1} \eqqcolon \gamma', &\quad 1/\gamma^2 &\longmapsto \frac{1-\gamma^2}{1+\gamma^2} = -\gamma'.\\[-4.2mm]
\end{alignat*}
Hence, the curves $H_1$ and $H_2$ are isomorphic to\vspace{1mm}
\begin{align*}
    H_1: & \ Y^2 = (X^2-\alpha'^2)(X^2-\beta'^2)(X^2-\gamma'^2),\\[-1.3mm]
    H_2: & \ Y^2 = (X^2-1)(X^2-\alpha'^2)(X^2-\beta'^2)(X^2-\gamma'^2),\\[-4mm]
\end{align*}
respectively.
Applying \cite[Theorem 1]{KT} again, the Jacobian of $H_1$ is Richelot isogenous to the product of two elliptic curves\vspace{0.5mm}
\begin{align*}
    E_1: & \ Y^2 = (X-\alpha'^2)(X-\beta'^2)(X-\gamma'^2),\\[-1.3mm]
    E_2: & \ Y^2 = X(X-\alpha'^2)(X-\beta'^2)(X-\gamma'^2).\\[-3.7mm]
\end{align*}
Also, the Jacobian of $H_2$ is Richelot isogenous to the product of the elliptic curve\vspace{0.7mm}
\[
    E_3: Y^2 = (X-1)(X-\alpha'^2)(X-\beta'^2)(X-\gamma'^2)\vspace{1.3mm}
\]
and the Jacobian of the genus-2 curve\vspace{0.8mm}
\[
    C: Y^2 = X(X-1)(X-\alpha'^2)(X-\beta'^2)(X-\gamma'^2).\vspace{1.4mm}
\]
Hence, we obtain the Jacobian decomposition ${\rm Jac}(H_{a,b,c}) \sim E_1 \times E_2 \times E_3 \times {\rm Jac}(C)$.
Finally, we define the values $\lambda,\mu,\nu \notin \{0,1\}$ such that\vspace{0.6mm}
\begin{alignat*}{5}
    \lambda &\coloneqq \alpha'^2 &&= \biggl(\frac{\alpha^2-1}{\alpha^2+1}\biggr)^{\!2} &&= \frac{\alpha^2-2+1/\alpha^2}{\alpha^2+2+1/\alpha^2} &&= -\frac{2+a}{2-a},\\[-0.5mm]
    \mu &\coloneqq \beta'^2 &&= \biggl(\frac{\beta^2-1}{\beta^2+1}\biggr)^{\!2} &&= \frac{\beta^2-2+1/\beta^2}{\beta^2+2+1/\beta^2} &&= -\frac{2+b}{2-b},\\[-0.5mm]
    \nu &\coloneqq \gamma'^2 &&= \biggl(\frac{\gamma^2-1}{\gamma^2+1}\biggr)^{\!2} &&= \frac{\gamma^2-2+1/\gamma^2}{\gamma^2+2+1/\gamma^2} &&= -\frac{2+c}{2-c},\\[-4.5mm]
\end{alignat*}
which completes the proof.
\end{proof}
\end{theorem}\vspace{-2mm}

The following corollary plays a central role in our algorithm (Algorithm \ref{alg:type4-1}) which will be described in Section \ref{sec5}.\vspace{-4.8mm}
\begin{corollary}
The curve $H_{a,b,c}$ is superspecial if and only if $E_1,E_2,E_3$ are all supersingular and $C$ is superspecial.
Moreover, if $H_{a,b,c}$ is superspecial, then $a,b,c$ all lie in
$\mathbb{F}_{p^2}$.\vspace{-1.2mm}
\begin{proof}
It follows from the proof of Theorem \ref{thm:jacof4-1} that there exists an isogeny\vspace{0.9mm}
\[
    \phi: {\rm Jac}(H_{a,b,c}) \longrightarrow E_1 \times E_2 \times E_3 \times {\rm Jac}(C)\vspace{2.2mm}
\]
whose degree is a power of $2$, which implies the former assertion.
Moreover, the latter assertion follows from the superspeciality of $C: y^2 = x(x-1)(x-\lambda)(x-\mu)(x-\nu)$.
Indeed, it follows from \cite[Main Theorem A]{Ohashi} that\vspace{0.6mm}
\[
    1-\lambda = \frac{4}{2-a},\ \ 1-\mu = \frac{4}{2-b},\ \ 1-\nu = \frac{4}{2-c}\vspace{1.3mm}
\]
all belong to $\mathbb{F}_{p^2}$.
This implies that $a,b,c$ also belong to $\mathbb{F}_{p^2}$, as desired.
\end{proof}
\end{corollary}

\section{ Detailed analysis for individual curve types}\label{sec4}
This section provides a more detailed analysis of each type classified in the latter part of Section \ref{sec2}.
More specifically, the following subsections describe several properties of hyperelliptic genus-5 curves $H$ that satisfy the following conditions.\vspace{-0.1mm}
\begin{itemize}
\item Subsection 4.1:\ \,The case where ${\rm Aut}(H) \cong \mathbf{C}_2^3$, that is, $H$ is of Type {\bf 4-1}.\vspace{0.2mm}
\item Subsection 4.2:\ \,The case where ${\rm Aut}(H) \supset \mathbf{C}_2^2 \rtimes \mathbf{C}_4$.\vspace{0.2mm}
\item Subsection 4.3:\ \,The case where ${\rm Aut}(H) \supset \mathbf{C}_2 \times \mathbf{D}_{12}$.\vspace{0.2mm}
\item Subsection 4.4:\ \,The case where ${\rm Aut}(H) \supset \mathbf{C}_2 \times \mathbf{A}_4$.
\end{itemize}
In what follows, we say that the first case is the \emph{generic case}, and the other three cases are \emph{special cases}.
The main goals of this section are to provide a criterion for deciding whether two curves are isomorphic or not in the generic case, and to show that in each\\special case the Jacobian of $H$ is completely decomposable (i.e., isogenous to a product of five elliptic curves).
These results will be employed in our algorithms of Section \ref{sec5} to enumerate superspecial curves in each case.

\subsection{ Generic case: ${\rm Aut}(H) \cong \mathbf{C}_2^3$}\label{sec4-1}
The aim of this subsection is to prove Theorem \ref{thm:isomof4-1} below, which says that the converse of Proposition \ref{prp:isomof4-1} holds, under the assumption that the automorphism groups of $H_{a,b,c}$ and $H_{a',b',c'}\hspace{-0.1mm}$ are isomorphic to $\mathbf{C}_2^3$.
Under this assumption, the reduced automorphism groups of these curves are explicitly given by\vspace{-2mm}
\begin{alignat*}{2}
    \overline{{\rm Aut}}(H_{a,b,c}) &= \{{\rm id},\sigma,\tau,\sigma \hspace{-0.2mm}\circ \tau\}, \quad && \text{with }\,\sigma(x) = -x,\ \,\tau(x) = 1/x,\\[-0.8mm]
    \overline{{\rm Aut}}(H_{a',b',c'}) &= \{{\rm id},\sigma',\tau',\sigma' \hspace{-0.2mm}\circ \tau'\}, \quad && \text{with }\,\sigma'(x) = -x,\ \,\tau'(x) = 1/x,\\[-6.6mm]
\end{alignat*}
as already described in the latter part of Section~\ref{sec2}.
As a preparation for the proof, we first prove the following fundamental lemma:\vspace{-4.7mm}
\begin{lemma}\label{lem:scalar}
Let $s_1,s_2,s_3,t_1,t_2,t_3 \in k^*\hspace{-0.5mm}$ such that\vspace{0.2mm}
\[
    \#\{\pm s_1,\pm 1/s_1,\pm s_2,\pm 1/s_2,\pm s_3,\pm 1/s_3\} = \#\{\pm t_1,\pm 1/t_1,\pm t_2,\pm 1/t_2,\pm t_3,\pm 1/t_3\} = 12.\vspace{1.3mm}
\]
If there exists $\lambda \in k^*\hspace{-0.5mm}$ such that\vspace{0.2mm}
\[
    \{\pm \lambda s_1,\pm\lambda/s_1,\pm\lambda s_2,\pm\lambda/s_2,\pm\lambda s_3,\pm\lambda/s_3\} = \{\pm t_1,\pm 1/t_1,\pm t_2,\pm 1/t_2,\pm t_3,\pm 1/t_3\},\vspace{1.3mm}
\]
then $\lambda^4 = 1$.\vspace{-1mm}
\begin{proof}
By squaring each element of the sets on both sides, we obtain the equality\vspace{0.7mm}
\begin{equation}\label{eq:square-set}
    \{\lambda^2 s_1^2,\lambda^2/s_1^2,\lambda^2 s_2^2,\lambda^2/s_2^2,\lambda^2 s_3^2,\lambda^2/s_3^2\}
    =
    \{t_1^2,1/t_1^2,t_2^2,1/t_2^2,t_3^2,1/t_3^2\}\vspace{1.3mm}.
\end{equation}
Comparing the degree-$2$ elementary symmetric polynomials of the sets on both sides of \eqref{eq:square-set}, we obtain the equality\vspace{0.5mm}
\begin{equation}\label{eq:deg2-sym}
    \lambda^4 (s_1^2+1/s_1^2)(s_2^2+1/s_2^2)(s_3^2+1/s_3^2)
    =
    (t_1^2+1/t_1^2)(t_2^2+1/t_2^2)(t_3^2+1/t_3^2).\vspace{1.4mm}
\end{equation}
Similarly, comparing the degree-$4$ elementary symmetric polynomials of the sets on both sides of \eqref{eq:square-set}, we obtain the equality\vspace{0.5mm}
\begin{equation}\label{eq:deg4-sym}
    \lambda^8 (s_1^2+1/s_1^2)(s_2^2+1/s_2^2)(s_3^2+1/s_3^2)
    =
    (t_1^2+1/t_1^2)(t_2^2+1/t_2^2)(t_3^2+1/t_3^2).\vspace{1.4mm}
\end{equation}
Comparing the equalities \eqref{eq:deg2-sym} and \eqref{eq:deg4-sym}, we immediately conclude that $\lambda^4 = 1$.
\end{proof}
\end{lemma}

We suppose that an isomorphism $\psi: H_{a,b,c} \rightarrow H_{a',b',c'}$ exists, and denote by $\psi$ the induced automorphism of $\mathbb{P}^1\hspace{-0.2mm}$ by abuse of notation.
Since $\psi \circ \sigma \circ \psi^{-1} \in \overline{{\rm Aut}}(H_{a',b',c'})$ has order 2, one of the following conditions must hold:\vspace{-0.7mm}
\begin{enumerate}
\item[(i)] $\psi \circ \sigma \circ \psi^{-1} = \sigma'$, that is, $\psi(-x) = -\psi(x)$ for all $x \in \mathbb{P}^1$.
\item[(ii)] $\psi \circ \sigma \circ \psi^{-1} = \tau'$, that is, $\psi(-x) = 1/\psi(x)$ for all $x \in \mathbb{P}^1$.
\item[(iii)] $\psi \circ \sigma \circ \psi^{-1} = \sigma' \hspace{-0.2mm}\circ \tau'$, that is, $\psi(-x) = -1/\psi(x)$ for all $x \in \mathbb{P}^1$.\vspace{0.4mm}
\end{enumerate}
In what follows, we prove lemmas by considering each case separately.
For the proofs, we rewrite the defining polynomials of $H_{a,b,c}$ and $H_{a',b',c'}$ as\vspace{-1.8mm}
\begin{align*}
    H_{a,b,c}: y^2 &= (x^2-\alpha^2)(x^2-1/\alpha^2)(x^2-\beta^2)(x^2-1/\beta^2)(x^2-\gamma^2)(x^2-1/\gamma^2),\\[-1mm]
    H_{a',b',c'}: y^2 &= (x^2-\alpha'^2)(x^2-1/\alpha'^2)(x^2-\beta'^2)(x^2-1/\beta'^2)(x^2-\gamma'^2)(x^2-1/\gamma'^2),\\[-6.1mm]
\end{align*}
where $\alpha,\beta,\gamma,\alpha',\beta',\gamma' \in k^*\hspace{-0.3mm}$ with $\alpha^2+1/\alpha^2 = -a,\ \beta^2+1/\beta^2 = -b,\ \gamma^2+1/\gamma^2 = -c$, and so on.\vspace{-4.7mm}
\begin{lemma}\label{lem:case1}
If $\psi(-x) = -\psi(x)$ for all $x \in \mathbb{P}^1$, then $\{a',b',c'\} = \{a,b,c\}$ or $\{-a,-b,-c\}$.\vspace{-0.6mm}
\begin{proof}
By the assumption, the isomorphism $\psi$ on $\mathbb{P}^1$ must be of the form\vspace{0.3mm}
\[
    \psi(x) = \lambda x\ \text{ or }\ \psi(x) = \lambda/x\vspace{1.2mm}
\]
for some $\lambda \neq 0$.
Since $\psi$ maps the set of branch points of $H_{a,b,c}$ to that of $H_{a',b',c'}$, we have\vspace{0.4mm}
\begin{align*}
    \{\pm\lambda\alpha,\pm\lambda/\alpha,\pm\lambda\beta,\pm\lambda/\beta,\pm\lambda\gamma,\pm\lambda/\gamma\} &= \psi(\{\pm\alpha,\pm1/\alpha,\pm\beta,\pm1/\beta,\pm\gamma,\pm1/\gamma\})\\[-1.2mm]
    &=\{\pm\alpha',\pm1/\alpha',\pm\beta',\pm1/\beta',\pm\gamma',\pm1/\gamma'\}.\\[-4.6mm]
\end{align*}
It then follows from Lemma~\ref{lem:scalar} that $\lambda^4 = 1$.
Moreover, squaring each element of the sets on both sides yields\vspace{0.4mm}
\begin{equation}\label{eq:case1}
  \{\lambda^2\alpha^2,\lambda^2/\alpha^2,\lambda^2\beta^2,\lambda^2/\beta^2,\lambda^2\gamma^2,\lambda^2/\gamma^2\} = \{\alpha'^2,1/\alpha'^2,\beta'^2,1/\beta'^2,\gamma'^2,1/\gamma'^2\}.\vspace{1.1mm}
\end{equation}
We now consider two cases depending on whether $\lambda^2 = 1$ or
$\lambda^2 = -1$.\vspace{-0.9mm}
\begin{itemize}
\item If $\lambda^2 = 1$, then the sets on both sides of \eqref{eq:case1} can be partitioned into\vspace{-2.1mm}
\begin{align*}
    \{\{\alpha^2,1/\alpha^2\},\{\beta^2,1/\beta^2\},&\hspace{0.5mm}\{\gamma^2,1/\gamma^2\}\}\\[-0.5mm]
    &= \{\{\alpha'^2,1/\alpha'^2\},\{\beta'^2,1/\beta'^2\},\{\gamma'^2,1/\gamma'^2\}\}.\\[-6.2mm]
\end{align*}
Hence, we obtain\vspace{-2.6mm}
\begin{align*}
    \{a',b',c'\} &= \{-(\alpha'^2+1/\alpha'^2),-(\beta'^2+1/\beta'^2),-(\gamma'^2+1/\gamma'^2)\}\\[-1mm]
    &= \{-(\alpha^2+1/\alpha^2),-(\beta^2+1/\beta^2),-(\gamma^2+1/\gamma^2)\} = \{a,b,c\}.\\[-6.6mm]
\end{align*}
\item If $\lambda^2 = -1$, then the sets on both sides of
\eqref{eq:case1} can be partitioned as\vspace{-2.2mm}
\begin{align*}
    \{\{-\alpha^2,-1/\alpha^2\},\{-\beta^2,-1/\beta^2\},&\hspace{0.5mm}\{-\gamma^2,-1/\gamma^2\}\}\\[-0.5mm]
    &= \{\{\alpha'^2,1/\alpha'^2\},\{\beta'^2,1/\beta'^2\},\{\gamma'^2,1/\gamma'^2\}\}.\\[-6.2mm]
\end{align*}
Hence, we obtain\vspace{-2.6mm}
\begin{align*}
    \{a',b',c'\} &= \{-(\alpha'^2+1/\alpha'^2),-(\beta'^2+1/\beta'^2),-(\gamma'^2+1/\gamma'^2)\}\\[-1mm]
    &= \{\alpha^2+1/\alpha^2,\beta^2+1/\beta^2,\gamma^2+1/\gamma^2\} = \{-a,-b,-c\}\\[-6.6mm]
\end{align*}
\end{itemize}
Therefore, the proof is completed.
\end{proof}
\end{lemma}

\begin{lemma}\label{lem:case2}
If $\psi(-x) = 1/\psi(x)$ for all $x \in \mathbb{P}^1$, then\vspace{0.7mm}
\[
    \{a',b',c'\} = \biggl\{\frac{12-2a}{2+a},\frac{12-2b}{2+b},\frac{12-2c}{2+c}\biggr\} \,\text{ or }\, \biggl\{\frac{12+2a}{2-a},\frac{12+2b}{2-b},\frac{12+2c}{2-c}\biggr\}.\vspace{-2.8mm}
\]
\begin{proof}
By the assumption, the isomorphism $\psi$ on $\mathbb{P}^1$ must be of the form\vspace{0.2mm}
\begin{equation}\label{eq:psi-case2}
    \psi(x) = \frac{x-\lambda}{x+\lambda} \quad \text{or} \quad \psi(x) = -\frac{x-\lambda}{x+\lambda}\vspace{0.7mm}
\end{equation}
for some $\lambda \neq 0$.
For the former isomorphism $\psi$, we obtain\vspace{0.3mm}
\begin{align}\label{eq:case2}
    \biggl\{
    \frac{\alpha\mp\lambda}{\alpha\pm\lambda},\frac{1\mp\lambda\alpha}{1\pm\lambda\alpha},
    \frac{\beta\mp\lambda}{\beta\pm\lambda},\frac{1\mp\lambda\beta}{1\pm\lambda\beta},
    \frac{\gamma\mp\lambda}{\gamma\pm\lambda},&\frac{1\mp\lambda\gamma}{1\pm\lambda\gamma}
    \biggr\}\nonumber\\[-0.3mm] &= \psi(\{\pm\alpha,\pm1/\alpha,\pm\beta,\pm1/\beta,\pm\gamma,\pm1/\gamma\})\nonumber\\[-0.7mm]
    &= \{\pm\alpha',\pm1/\alpha',\pm\beta',\pm1/\beta',\pm\gamma',\pm1/\gamma'\}.\\[-5.4mm]\nonumber
\end{align}
Here, the Möbius transformation $x \mapsto \frac{x-1}{x+1}$ yields\vspace{-0.5mm}
\begin{align*}
    \{\pm\lambda\alpha,\pm\lambda/\alpha,\pm\lambda\beta,&\pm\lambda/\beta,\pm\lambda\gamma,\pm\lambda/\gamma\}\\[-0.3mm]
    &= \left\{\pm\frac{\alpha'-1}{\alpha'+1},\pm\frac{\alpha'+1}{\alpha'-1},\pm\frac{\beta'-1}{\beta'+1},\pm\frac{\beta'+1}{\beta'-1},\pm\frac{\gamma'-1}{\gamma'+1},\pm\frac{\gamma'+1}{\gamma'-1}\right\}.\\[-5.7mm]
\end{align*}
Then, it follows from Lemma \ref{lem:scalar} that $\lambda^4=1$.
In the following, let us consider separate cases depending on whether $\lambda^2 = 1$ or $\lambda^2 = -1$.\vspace{-1mm}
\begin{itemize}
\item If $\lambda^2 = 1$, squaring each element of the sets on both sides of \eqref{eq:case2} yields\vspace{-2.2mm}
\begin{align*}
    \biggl\{
    \!\left\{\Bigl(\frac{\alpha+1}{\alpha-1}\Bigr)^{\!2},\Bigl(\frac{\alpha-1}{\alpha+1}\Bigr)^{\!2}\right\},
    \left\{\Bigl(\frac{\beta+1}{\beta-1}\Bigr)^{\!2},\Bigl(\frac{\beta-1}{\beta+1}\Bigr)^{\!2}\right\},
    \left\{\Bigl(\frac{\gamma+1}{\gamma-1}\Bigr)^{\!2},\Bigl(\frac{\gamma-1}{\gamma+1}\Bigr)^{\!2}\right\}\!\biggr\} \\[-0.5mm]= \{\{\alpha'^2,1/\alpha'^2\},\{\beta'^2,1/\beta'^2\},\{\gamma'^2,1/\gamma'^2\}\}.\\[-7.2mm]
\end{align*}
Hence, we obtain\vspace{-2.7mm}
\begin{align*}
    \{a',b',c'\} &= \{-(\alpha'^2+1/\alpha'^2),-(\beta'^2+1/\beta'^2),-(\gamma'^2+1/\gamma'^2)\}\\[-0.5mm]
    &= \left\{-\Bigl(\frac{\alpha-1}{\alpha+1}\Bigr)^{\!\!2}\!-\Bigl(\frac{\alpha+1}{\alpha-1}\Bigr)^{\!\!2},-\Bigl(\frac{\beta-1}{\beta+1}\Bigr)^{\!\!2}\!-\Bigl(\frac{\beta+1}{\beta-1}\Bigr)^{\!\!2},-\Bigl(\frac{\gamma-1}{\gamma+1}\Bigr)^{\!\!2}\!-\Bigl(\frac{\gamma+1}{\gamma-1}\Bigr)^{\!\!2}\right\}\\[-0.7mm]
    &= \biggl\{\frac{12-2a}{2+a},\frac{12-2b}{2+b},\frac{12-2c}{2+c}\biggr\}.\\[-7.8mm]
\end{align*}
\item If $\lambda^2 = -1$, squaring each element of the sets on both sides of \eqref{eq:case2} yields\vspace{-2.2mm}
\begin{align*}
    \biggl\{
    \!\left\{\Bigl(\frac{\alpha+i}{\alpha-i}\Bigr)^{\!2},\Bigl(\frac{\alpha-i}{\alpha+i}\Bigr)^{\!2}\right\},
    \left\{\Bigl(\frac{\beta+i}{\beta-i}\Bigr)^{\!2},\Bigl(\frac{\beta-i}{\beta+i}\Bigr)^{\!2}\right\},
    \left\{\Bigl(\frac{\gamma+i}{\gamma-i}\Bigr)^{\!2},\Bigl(\frac{\gamma-i}{\gamma+i}\Bigr)^{\!2}\right\}\!\biggr\} \\= \{\{\alpha'^2,1/\alpha'^2\},\{\beta'^2,1/\beta'^2\},\{\gamma'^2,1/\gamma'^2\}\}.\\[-7.2mm]
\end{align*}
Hence, we obtain\vspace{-2.7mm}
\begin{align*}
    \{a',b',c'\} &= \{-(\alpha'^2+1/\alpha'^2),-(\beta'^2+1/\beta'^2),-(\gamma'^2+1/\gamma'^2)\}\\[-0.5mm]
    &= \left\{-\Bigl(\frac{\alpha-i}{\alpha+i}\Bigr)^{\!\!2}\!-\Bigl(\frac{\alpha+i}{\alpha-i}\Bigr)^{\!\!2},-\Bigl(\frac{\beta-i}{\beta+i}\Bigr)^{\!\!2}\!-\Bigl(\frac{\beta+i}{\beta-i}\Bigr)^{\!\!2},-\Bigl(\frac{\gamma-i}{\gamma+i}\Bigr)^{\!\!2}\!-\Bigl(\frac{\gamma+i}{\gamma-i}\Bigr)^{\!\!2}\right\}\\[-0.5mm]
    &= \biggl\{\frac{12+2a}{2-a},\frac{12+2b}{2-b},\frac{12+2c}{2-c}\biggr\}.\\[-7.6mm]
\end{align*}
\end{itemize}
The claim can also be verified in a similar way for the latter isomorphism $\psi$ in \eqref{eq:psi-case2}.
Therefore, the proof is completed.
\end{proof}
\end{lemma}

\begin{lemma}\label{lem:case3}
If $\psi(-x) = -1/\psi(x)$ for all $x \in \mathbb{P}^1$, then\vspace{0.7mm}
\[
    \{a',b',c'\} = \biggl\{-\frac{12-2a}{2+a},-\frac{12-2b}{2+b},-\frac{12-2c}{2+c}\biggr\} \,\text{ or }\, \biggl\{-\frac{12+2a}{2-a},-\frac{12+2b}{2-b},-\frac{12+2c}{2-c}\biggr\}.\vspace{-2.8mm}
\]
\begin{proof}
By the assumption, the isomorphism $\psi$ on $\mathbb{P}^1$ must be of the form\vspace{0.2mm}
\begin{equation}\label{eq:psi-case3}
    \psi(x) = i\frac{x-\lambda}{x+\lambda} \quad \text{or} \quad \psi(x) = -i\frac{x-\lambda}{x+\lambda}\vspace{0.7mm}
\end{equation}
for some $\lambda \neq 0$.
For the former isomorphism $\psi$, we obtain\vspace{0.3mm}
\begin{align}\label{eq:case3}
    \biggl\{
    i\frac{\alpha\mp\lambda}{\alpha\pm\lambda},i\frac{1\mp\lambda\alpha}{1\pm\lambda\alpha},
    i\frac{\beta\mp\lambda}{\beta\pm\lambda},i\frac{1\mp\lambda\beta}{1\pm\lambda\beta},
    &i\frac{\gamma\mp\lambda}{\gamma\pm\lambda},i\frac{1\mp\lambda\gamma}{1\pm\lambda\gamma}
    \biggr\}\nonumber\\[-0.3mm] &= \psi(\{\pm\alpha,\pm1/\alpha,\pm\beta,\pm1/\beta,\pm\gamma,\pm1/\gamma\})\nonumber\\[-0.5mm]
    &= \{\pm\alpha',\pm1/\alpha',\pm\beta',\pm1/\beta',\pm\gamma',\pm1/\gamma'\}.\\[-5.4mm]\nonumber
\end{align}
Here, the Möbius transformation $x \mapsto \frac{x-i}{x+i}$ yields\vspace{-0.5mm}
\begin{align*}
    \{\pm\lambda\alpha,\pm\lambda/\alpha,\pm\lambda\beta,&\pm\lambda/\beta,\pm\lambda\gamma,\pm\lambda/\gamma\}\\[-0.3mm]
    &= \left\{\pm\frac{\alpha'-i}{\alpha'+i},\pm\frac{\alpha'+i}{\alpha'-i},\pm\frac{\beta'-i}{\beta'+i},\pm\frac{\beta'+i}{\beta'-i},\pm\frac{\gamma'-i}{\gamma'+i},\pm\frac{\gamma'+i}{\gamma'-i}\right\}.\\[-5.7mm]
\end{align*}
Then, it follows from Lemma \ref{lem:scalar} that $\lambda^4=1$.
In the following, let us consider separate cases depending on whether $\lambda^2 = 1$ or $\lambda^2 = -1$.\vspace{-1mm}
\begin{itemize}
\item If $\lambda^2 = 1$, squaring each element of the sets on both sides of \eqref{eq:case3} yields\vspace{-2.2mm}
\begin{align*}
    \biggl\{
    \!\left\{-\Bigl(\frac{\alpha+1}{\alpha-1}\Bigr)^{\!2},-\Bigl(\frac{\alpha-1}{\alpha+1}\Bigr)^{\!2}\right\},
    \left\{-\Bigl(\frac{\beta+1}{\beta-1}\Bigr)^{\!2},-\Bigl(\frac{\beta-1}{\beta+1}\Bigr)^{\!2}\right\},
    \left\{\Bigl(-\frac{\gamma+1}{\gamma-1}\Bigr)^{\!2},\Bigl(-\frac{\gamma-1}{\gamma+1}\Bigr)^{\!2}\right\}\!\biggr\} \\[-0.5mm]= \{\{\alpha'^2,1/\alpha'^2\},\{\beta'^2,1/\beta'^2\},\{\gamma'^2,1/\gamma'^2\}\}.\\[-7.2mm]
\end{align*}
Hence, we obtain\vspace{-2.7mm}
\begin{align*}
    \{a',b',c'\} &= \{-(\alpha'^2+1/\alpha'^2),-(\beta'^2+1/\beta'^2),-(\gamma'^2+1/\gamma'^2)\}\\[-0.5mm]
    &= \left\{\Bigl(\frac{\alpha-1}{\alpha+1}\Bigr)^{\!\!2}\!+\Bigl(\frac{\alpha+1}{\alpha-1}\Bigr)^{\!\!2},\Bigl(\frac{\beta-1}{\beta+1}\Bigr)^{\!\!2}\!+\Bigl(\frac{\beta+1}{\beta-1}\Bigr)^{\!\!2},\Bigl(\frac{\gamma-1}{\gamma+1}\Bigr)^{\!\!2}\!+\Bigl(\frac{\gamma+1}{\gamma-1}\Bigr)^{\!\!2}\right\}\\[-0.5mm]
    &= \biggl\{-\frac{12-2a}{2+a},-\frac{12-2b}{2+b},-\frac{12-2c}{2+c}\biggr\}.\\[-7.8mm]
\end{align*}
\item If $\lambda^2 = -1$, squaring each element of the sets on both sides of \eqref{eq:case3} yields\vspace{-2.2mm}
\begin{align*}
    \biggl\{
    \!\left\{-\Bigl(\frac{\alpha+i}{\alpha-i}\Bigr)^{\!2},-\Bigl(\frac{\alpha-i}{\alpha+i}\Bigr)^{\!2}\right\},
    \left\{-\Bigl(\frac{\beta+i}{\beta-i}\Bigr)^{\!2},-\Bigl(\frac{\beta-i}{\beta+i}\Bigr)^{\!2}\right\},
    \left\{-\Bigl(\frac{\gamma+i}{\gamma-i}\Bigr)^{\!2},-\Bigl(\frac{\gamma-i}{\gamma+i}\Bigr)^{\!2}\right\}\!\biggr\} \\[-0.5mm]= \{\{\alpha'^2,1/\alpha'^2\},\{\beta'^2,1/\beta'^2\},\{\gamma'^2,1/\gamma'^2\}\}.\\[-7.2mm]
\end{align*}
Hence, we obtain\vspace{-2.7mm}
\begin{align*}
    \{a',b',c'\} &= \{-(\alpha'^2+1/\alpha'^2),-(\beta'^2+1/\beta'^2),-(\gamma'^2+1/\gamma'^2)\}\\[-0.5mm]
    &= \left\{\Bigl(\frac{\alpha-i}{\alpha+i}\Bigr)^{\!\!2}\!+\Bigl(\frac{\alpha+i}{\alpha-i}\Bigr)^{\!\!2},\Bigl(\frac{\beta-i}{\beta+i}\Bigr)^{\!\!2}\!+\Bigl(\frac{\beta+i}{\beta-i}\Bigr)^{\!\!2},\Bigl(\frac{\gamma-i}{\gamma+i}\Bigr)^{\!\!2}\!+\Bigl(\frac{\gamma+i}{\gamma-i}\Bigr)^{\!\!2}\right\}\\[-0.5mm]
    &= \biggl\{-\frac{12+2a}{2-a},-\frac{12+2b}{2-b},-\frac{12+2c}{2-c}\biggr\}\\[-7.6mm]
\end{align*}
\end{itemize}
The claim can also be verified in a similar way for the latter isomorphism $\psi$ in \eqref{eq:psi-case3}.
Hence, the proof is completed.
\end{proof}
\end{lemma}

By summarizing the discussion so far, we finally obtain the following theorem:\vspace{-4.5mm}
\begin{theorem}\label{thm:isomof4-1}
Assume that ${\rm Aut}(H_{a,b,c}) \cong \mathbf{C}_2^3$.
Then, the two curves $H_{a,b,c}$ and $H_{a',b',c'}$ are isomorphic to each other if and only if either of the following conditions holds:\vspace{-0.7mm}
\begin{enumerate}
\item $\{a',b',c'\} = \{a,b,c\}$,
\item $\{a',b',c'\} = \{-a,-b,-c\}$,
\item $\{a',b',c'\} = \{\frac{12-2a}{2+a},\frac{12-2b}{2+b},\frac{12-2c}{2+c}\}$,\vspace{-0.3mm}
\item $\{a',b',c'\} = \{-\frac{12-2a}{2+a},-\frac{12-2b}{2+b},-\frac{12-2c}{2+c}\}$,\vspace{-0.3mm}
\item $\{a',b',c'\} = \{\frac{12+2a}{2-a},\frac{12+2b}{2-b},\frac{12+2c}{2-c}\}$,\vspace{-0.3mm}
\item $\{a',b',c'\} = \{-\frac{12+2a}{2-a},-\frac{12+2b}{2-b},-\frac{12+2c}{2-c}\}$.\vspace{-2.7mm}
\end{enumerate}
\begin{proof}
This follows from Proposition \ref{prp:isomof4-1} together with Lemmas \ref{lem:case1}, \ref{lem:case2}, and \ref{lem:case3}.
\end{proof}
\end{theorem}\vspace{-4.5mm}

\subsection{ Special case: ${\rm Aut}(H) \supset \mathbf{C}_2^2 \rtimes \mathbf{C}_4$}
In this subsection, we consider a hyperelliptic genus-5 curve\vspace{-1.9mm}
\begin{equation}\label{eq:type7}
    H: y^2 = (x^4+1)(x^8-Ax^4+1)\vspace{-0.9mm}
\end{equation}
with $A \in k\!\smallsetminus\!\{\pm2\}$, which is clearly isomorphic to the curve $y^2 = (x^4-1)(x^8+Ax^4+1)$ via the transformations $x^4 \mapsto\hspace{-0.2mm} -x^4$ and $y^2 \mapsto\hspace{-0.2mm} -y^2$.
It follows from Table \ref{table:classification} that $H$ is of Type {\bf 7}, \hspace{-0.3mm}{\bf 11}, or {\bf 15} (we note that $H$ is of Type {\bf 11} if $A=34$, and of Type {\bf 15} if $A=1$).
This curve $H$ is isomorphic to $H_{a,b,c}$ defined in \eqref{eq:Habc}, where $a,b$, and $c$ are the distinct roots of the cubic equation $t^3-(2+A)t = 0$.
Indeed, the right-hand side of \eqref{eq:Habc} can be expanded as\vspace{-2.4mm}
\begin{align*}
    \begin{split}
    x^{12} + (a+b+c)x^{10} &+ (ab+bc+ca+3)x^8 + (abc+2a+2b+2c)x^6\\[-1.2mm]
    & + (ab+bc+ca+3)x^4 + (a+b+c)x^2 + 1,
    \end{split}\\[-5.8mm]\nonumber
\end{align*}
which is equal to $x^{12} + (1-A)x^8 + (1-A)x^4 + 1 = (x^4+1)(x^8-Ax^4+1)$, since\vspace{-1.4mm}
\begin{equation}\label{eq:abc-type7}
    a+b+c = 0, \quad ab+bc+ca = -(2+A), \quad abc = 0\vspace{-0.5mm}
\end{equation}
holds from the way $a,b,c$ are chosen.
Therefore, the Jacobian decomposition of $H$ can be described using Theorem \ref{thm:jacof4-1}.\vspace{-4.7mm}
\begin{proposition}\label{prp:jacof7}
The Jacobian of $H$ given in equation \eqref{eq:type7} is decomposed into the product of five elliptic curves $E_1,E_2,E_3,E_4,E_5$, where we define\vspace{0.5mm}
\begin{align*}
    E_1: Y^2 &= (X+1)(X-\delta)(X-1/\delta),\\[-1mm]
    E_2: Y^2 &= X(X+1)(X-\delta)(X-1/\delta),\\[-1mm]
    E_3: Y^2 &= (X-1)(X+1)(X-\delta)(X-1/\delta),\\[-1mm]
    E_4: Y^2 &= X(X-1)\{X+(\delta-(\delta^2-1)^{1/2})^2\},\\[-1mm]
    E_5: Y^2 &= X(X-1)\{X+(\delta+(\delta^2-1)^{1/2})^2\}\\[-4.3mm]
\end{align*}
with\vspace{-0.7mm}
\begin{equation}\label{eq:deltaof7}
    \delta + 1/\delta = -\frac{12+2A}{2-A}.\vspace{1.5mm}
\end{equation}
Therefore, the curve $H$ is superspecial if and only if the five elliptic curves $E_1,E_2,E_3,E_4,E_5$ are all supersingular.
\begin{proof}
For the values $(\lambda,\mu,\nu)$ defined in \eqref{eq:lmn}, it follows from the relation \eqref{eq:abc-type7} that
\begin{align*}
    -(\lambda+\mu+\nu) &= \frac{2+a}{2-a} + \frac{2+b}{2-b} + \frac{2+c}{2-c}\\
    &= \frac{24-4(a+b+c)-2(ab+bc+ca)+3abc}{8-4(a+b+c)+2(ab+bc+ca)-abc} = \frac{14+A}{2-A},\\
    \lambda\mu+\mu\nu+\nu\lambda &= \frac{2+a}{2-a}\cdot \frac{2+b}{2-b} + \frac{2+b}{2-b}\cdot \frac{2+c}{2-c} + \frac{2+c}{2-c}\cdot \frac{2+a}{2-a}\\
    &= \frac{24+4(a+b+c)-2(ab+bc+ca)-3abc}{8-4(a+b+c)+2(ab+bc+ca)-abc} = \frac{14+A}{2-A},\\
    -\lambda\mu\nu &= \frac{2+a}{2-a}\cdot \frac{2+b}{2-b} \cdot \frac{2+c}{2-c}\\
    &= \frac{8+4(a+b+c)+2(ab+bc+ca)+abc}{8-4(a+b+c)+2(ab+bc+ca)-abc} = 1,\\[-5.3mm]
\end{align*}
whence\vspace{-0.7mm}
\begin{align*}
    (X-\lambda)(X-\mu)(X-\nu) &= X^3 - (\lambda+\mu+\nu)X^2 +(\lambda\mu + \mu\nu + \nu\lambda)X - \lambda\mu\nu\\[-0.3mm]
    &= X^3 + \frac{14+A}{2-A}X^2 + \frac{14+A}{2-A} + 1\\
    &= (X+1)\Bigl(X^2 + \frac{12+2A}{2-A}X+1 \Bigr) = (X+1)(X-\delta)(X-1/\delta).
\end{align*}
Therefore, applying Theorem~\ref{thm:jacof4-1}, one can see that the Jacobian of $H$ is decomposed into the product of $E_1,E_2,E_3$ and that of the genus-2 curve\vspace{0.3mm}
\begin{equation}\label{eq:g2-type7}
    C: Y^2 = X(X-1)(X+1)(X-\delta)(X-1/\delta).\vspace{1.4mm}
\end{equation}
This curve corresponds to Class~(3) in \cite[p.\,130]{IKO}, and thus its Jacobian is Richelot isogenous to $E_4 \hspace{0.2mm}\times E_5$ as shown in \cite[Remark 1.4]{IKO}.
From the above discussions, we obtain the Jacobian decomposition ${\rm Jac}(H) \sim E_1 \times E_2 \times E_3 \times E_4 \times E_5$, whose degree is a power of 2.
\end{proof}
\end{proposition}\vspace{-1.5mm}
The $j$-invariant of each elliptic curve in Proposition~\ref{prp:jacof7} can be computed as\vspace{-2.3mm}
\begin{align*}
    \begin{split}
    j(E_1) = j(E_2) &= 256\frac{(\delta^2-\delta+1)^3}{\delta^2(\delta-1)^2} = 16\frac{(A+14)^3}{(A-2)^2}, \quad j(E_3) = 1728, \text{ and }\\
    j(E_4) = j(E_5) &= 64\frac{(2\delta-1)^3(2\delta+1)^3}{\delta^2}.
    \end{split}\\[-6.5mm]
\end{align*}
Since $E_1 \cong\hspace{-0.1mm} E_2$ and $E_4 \cong\hspace{-0.1mm} E_5$, it suffices to check that $E_1,E_3$, and $E_4$ are supersingular in order to show that $H$ is superspecial.
Also, we obtain the following statement:\vspace{-4.5mm}
\begin{corollary}\label{cor:modof7}
If the curve $H\hspace{-0.1mm}$ given in \eqref{eq:type7} is superspecial, then $\hspace{0.3mm}p \equiv 3 \pmod{4}$ must hold.
In particular, the curve $H$ of Type {\bf 11} is superspecial if and only if $p \equiv 11 \pmod{12}$.\vspace{-1mm}
\begin{proof}
The former assertion follows from the well-known fact (cf. \cite[Example\,V\hspace{-0.3mm}.4.5]{Silverman}) that an\\ elliptic curve with $j$-invariant $\hspace{-0.2mm}1728$ is supersingular if and only if $p \equiv 3 \hspace{-0.3mm}\pmod{4}$.
Also, in the case where $A = 34$, we may take $\delta = 1/2$, and a direct computation shows that\vspace{0.8mm}
\[
    j(E_1) = j(E_2) = j(E_3) = 1728 \,\text{ and }\, j(E_4) = j(E_5) = 0.\vspace{2.2mm}
\]
Since it is also well-known (cf. \cite[Example\,V\hspace{-0.3mm}.4.4]{Silverman}) that an elliptic curve with $j$-invariant $0$ is\\ supersingular if and only if $p \equiv 2 \pmod{3}$, we obtain the latter assertion.
\end{proof}
\end{corollary}
The following proposition ensures that superspecial curves $H$ of the form \eqref{eq:type7} can be constructed using only arithmetic operations over $\mathbb{F}_{p^2}$.\vspace{-5.5mm}
\begin{proposition}\label{prop:rationalof7}
If the curve $H$ given in \eqref{eq:type7} is superspecial, then $\delta$ and $(\delta^2-1)^{1/2}\hspace{-0.2mm}$ belong to $\mathbb{F}_{p^2}$.\vspace{-0.9mm}
\begin{proof}
By the assumption, the genus-2 curve $C$ defined as in \eqref{eq:g2-type7} is also superspecial.
Then, thanks to \cite[Main Theorem A]{Ohashi}, we have that $\delta,\delta-1$, and $\delta+1$ are all squares in $\mathbb{F}_{p^2}\hspace{-0.3mm}$, which proves the assertion.
\end{proof}
\end{proposition}\vspace{-4.5mm}

\subsection{ Special case: ${\rm Aut}(H) \supset \mathbf{C}_2 \times \mathbf{D}_{12}$}
In this subsection, we consider a hyperelliptic genus-5 curve\vspace{-2mm}
\begin{equation}\label{eq:type9}
    H: y^2 = x^{12}+Ax^6+1\vspace{-0.8mm}
\end{equation}
with $A \in k\!\smallsetminus\!\{\pm 2\}$.
It follows from Table~\ref{table:classification} that $H$ is of Type {\bf 9} or {\bf 15} (we remark that if $A=0$, then $H$ is of Type {\bf 15}).
This curve $H$ is isomorphic to $H_{a,b,c}$ defined in (\ref{eq:Habc}), where $a,b$, and $c$ are distinct roots of the cubic equation $t^3-3t-A = 0$.
Indeed, the right-hand side of \eqref{eq:Habc}\ can be expanded as\vspace{-2.3mm}
\begin{align*}
    x^{12} + (a+b+c)x^{10} &+ (ab+bc+ca+3)x^8 + (abc+2a+2b+2c)x^6\nonumber\\[-1.2mm]
    & + (ab+bc+ca+3)x^4 + (a+b+c)x^2 + 1,\\[-6.2mm]\nonumber
\end{align*}
which is equal to $x^{12}+Ax^6+1$, since\vspace{-2.2mm}
\[
    a+b+c = 0, \quad ab+bc+ca = -3, \quad abc = A\vspace{-1mm}
\]
holds from the way $a,b,c$ are chosen.
Therefore, the Jacobian decomposition of $H$ can be described using Theorem~\ref{thm:jacof4-1}; \,however, we provide an alternative decomposition in the following proposition:\vspace{-4.7mm}
\begin{proposition}\label{prop:jacof9}
The Jacobian of $H$ given in equation \eqref{eq:type9} is decomposed into the product of five elliptic curves $E_1,E_2,E_3,E_4,E_4$, where we define\vspace{0.4mm}
\begin{align*}
    E_1: Y^2 &= X^3 + (3X+2-A)^2,\\[-1mm]
    E_2: Y^2 &= X^3 + (3X+2+A)^2,\\[-1mm]
    E_3: Y^2 &= X(X^2+AX+1),\\[-1mm]
    E_4: Y^2 &= X^3-3X+A.\\[-4.6mm]
\end{align*}
Therefore, the curve $H$ is superspecial if and only if the four elliptic curves $E_1,E_2,E_3,E_4$ are\\ all supersingular.\vspace{-1mm}
\begin{proof}
Applying \cite[Theorem 1]{KT}, we see that the Jacobian of $H$ is Richelot isogenous to the product of the Jacobians of the genus-2 curve $C_1$ and the genus-3 curve $C_2$, where\vspace{0.3mm}
\begin{align*}
    C_1: y^2 &= x^6+Ax^3+1,\\[-1.4mm]
    C_2: y^2 &= x(x^6+Ax^3+1).\\[-3.8mm]
\end{align*}
The Jacobian of $C_1$ is known (cf. \cite[Lemma 1]{DK}) to be  Richelot isogenous to $E_1 \times E_2$.
Also, it is known (cf. \cite[Proposition 2.1]{LLRS}) that the Jacobian of $C_2$ is isogenous to $E_3 \times E_4^2$, whose degree is a product of powers of 2 and 3.
From the above discussions, we obtain the Jacobian decomposition ${\rm Jac}(H) \sim E_1 \times E_2 \times E_3 \times E_4^2$, as desired.
\end{proof}
\end{proposition}

On the other hand, the superspeciality of $H$ can also be described in terms of some truncated Gaussian hypergeometric series, as we explain below.
First of all, we review the definition of a truncated Gaussian hypergeometric series:\vspace{-4.6mm}
\begin{definition}
For $a,b,c \in \mathbb{C}$ and $d \in \mathbb{N}$ with $-c \notin \mathbb{Z}_{\geq 0}$, we define\vspace{0.5mm}
\[
    G^{(d)}(a,b,c \mid z) \coloneqq \sum_{n=0}^d \frac{(a)_n(b)_n}{(c)_nn!}z^n,\vspace{1.5mm}
\]
where $(x)_n$ denotes the Pochhammer symbol.
If $a,b,c \in \mathbb{Q}$, then $G^{(d)}(a,b,c \mid z)$ can be regarded as a polynomial over $\mathbb{F}_p$, provided that the denominator of each coefficient is coprime to $p$.
\end{definition}\vspace{-2mm}

The curve $H$ given in equation \eqref{eq:type9} is isomorphic to\vspace{-1.2mm}
\begin{equation}\label{eq:deltaof9}
    Y^2 = (X^6-1)(X^6-\delta) \ \text{ with }\, \delta+1/\delta = A^2-2
\end{equation}
via the change of variables $X=\delta^{1/12}x\hspace{0.3mm}$ and $Y=\delta^{1/2}y$.
We remark that $\delta \neq 0,1$ from the assumption $A \neq \pm 2$.
We now give a characterization of the superspeciality of this curve, from which arithmetic properties of the parameter $\delta$ will follow:\vspace{-4.4mm}
\begin{proposition}\label{prop:lamof9}
Assume that $p \geq 11$.
Then, the following statements are true:\vspace{-0.5mm}
\begin{enumerate}
\item[{\rm (a)}] If $p \equiv 1 \pmod{6}$, then the curve $H$ is superspecial if and only if\vspace{-1.9mm}
\[
    \left\{\begin{array}{l}
        G^{((p-1)/6)}(1/2,1/6,2/3 \mid \delta) = 0,\\
        G^{((p-1)/3)}(1/2,1/3,5/6 \mid \delta) = 0,\\
        G^{((p-1)/2)}(1/2,1/2,1 \mid \delta) = 0.
    \end{array}\right.\vspace{-1.3mm}
\]
\item[{\rm (b)}] If $p \equiv 5 \pmod{6}$, then the curve $H$ is superspecial if and only if\vspace{-1.9mm}
\[
    \left\{\begin{array}{l}
        G^{((p-5)/6)}(1/2,5/6,4/3 \mid \delta) = 0,\\
        G^{((p-2)/3)}(1/2,2/3,7/6 \mid \delta) = 0,\\
        G^{((p-1)/2)}(1/2,1/2,1 \mid \delta) = 0.
    \end{array}\right.\vspace{-4.6mm}
\]
\end{enumerate}
\begin{proof}
This assertion follows immediately from \cite[Theorem 3.9]{Ohashi-RIMS} with $g=5$.
\end{proof}
\end{proposition}\vspace{-9.3mm}
\begin{proposition}\label{prop:rationalof9}
If $H$ is superspecial, then $\delta$ belongs to $\mathbb{F}_{p^2}$.\vspace{-1mm}
\begin{proof}
It follows from Proposition~\ref{prop:lamof9} that $G^{((p-1)/2)}(1/2,1/2,1 \mid \delta) = 0$, which is equivalent to the elliptic curve $y^2 = x(x-1)(x-\delta)$ in Legendre form being supersingular.
Hence, it is well-known (cf. \cite[Proposition 2.2]{AT}) that this implies $\delta \in \mathbb{F}_{p^2}$, as desired.
\end{proof}
\end{proposition}\vspace{-4mm}

\subsection{ Special case: ${\rm Aut}(H) \supset \mathbf{C}_2 \times \mathbf{A}_4$}
In this subsection, we consider a hyperelliptic genus-5 curve\vspace{-1.4mm}
\begin{equation}\label{eq:type10}
   H: y^2 = x^{12} - Ax^{10} - 33x^8 + 2Ax^6 - 33x^4 - Ax^2 + 1\vspace{-0.2mm}
\end{equation}
with $A \in k$.
The discriminant of the polynomial on the right-hand side of \eqref{eq:type10} can be computed as $2^{52}(A^2+108)^8$, and thus we obtain that $A^2+108 \neq 0$.
It follows from Table \ref{table:classification} that $H$ is of Type {\bf 10}, \hspace{-0.3mm}{\bf 11}, or {\bf 12} (we note that  $H$ is of Type {\bf 11} if $A=0$, and of Type {\bf 12} if $A^2 = 484/5$).

\begin{remark}
The curve $H$ given in \eqref{eq:type10} with $A^2 = 484/5$ is isomorphic to the curve\vspace{0.5mm}
\[
    H: y^2 = 5^3x^{12} - 22 \cdot 5^2x^{10} - 33 \cdot 5^2 x^8 + 44 \cdot 5x^6 - 33 \cdot 5x^4 - 22x^2 + 1\vspace{1.5mm}
\]
via the transformation $x^4 \mapsto\hspace{-0.2mm} \pm 5x^4$.
We executed the following code in \textsf{Magma}:\vspace{-2.2mm}
\begin{verbatim}
  K := Rationals();
  R<x> := PolynomialRing(K);
  f := 5^3*x^12 - 22*5^2*x^10 - 33*5^2*x^8 + 44*5*x^6 - 33*5*x^4 - 22*x^2 + 1;
  H := HyperellipticCurve(f);
  GroupName(GeometricAutomorphismGroup(H));
\end{verbatim}\vspace{-1.2mm}
As a result, we verify that ${\rm Aut}(H) \cong \mathbf{C}_2 \times \mathbf{A}_5$, and therefore $H$ is of Type~\textbf{12}.
Moreover, as in the discussion in \cite[Section 3.1]{LR}, this result holds over any algebraically closed field $k$ of characteristic $p \neq 2,3,5$.
\end{remark}\vspace{-2.2mm}

Also, the curve $H$ given in equation \eqref{eq:type10} is isomorphic to $H_{a,b,c}$ defined in \eqref{eq:Habc}, where $a,b$, and $c$ are the distinct roots of the cubic equation $t^3+At^2-36t-4A = 0$.
Indeed, the right-hand side of \eqref{eq:Habc} can be expanded as\vspace{-2.2mm}
\begin{align*}
    x^{12} + (a+b+c)x^{10} &+ (ab+bc+ca+3)x^8 + (abc+2a+2b+2c)x^6\nonumber\\[-1.2mm]
    & + (ab+bc+ca+3)x^4 + (a+b+c)x^2 + 1,\\[-6mm]\nonumber
\end{align*}
which is equal to $x^{12} - Ax^{10} - 33x^8 + 2Ax^6 - 33x^4 - Ax^2 + 1$, since\vspace{-2mm}
\begin{equation}\label{eq:abc-type10}
    a+b+c = -A, \quad ab+bc+ca = -36, \quad abc = 4A\vspace{-0.5mm}
\end{equation}
holds from the way $a,b,c$ are chosen.
Therefore, the Jacobian decomposition of $H$ can be described using Theorem \ref{thm:jacof4-1}.\vspace{-4.7mm}
\begin{proposition}\label{prp:jacof10}
The Jacobian of $H$ given in \eqref{eq:type10} is decomposed into the product of five elliptic curves $E_1,E_2,E_3,E_4,E_5$, where we define\vspace{0.5mm}
\begin{align*}
    E_1: Y^2 &= (X-\delta)\Bigl(X-\frac{\delta-1}{\delta}\Bigr)\hspace{-0.3mm}\Bigl(X-\frac{1}{1-\delta}\Bigr),\\[-0.5mm]
    E_2: Y^2 &= X(X-\delta)\Bigl(X-\frac{\delta-1}{\delta}\Bigr)\hspace{-0.3mm}\Bigl(X-\frac{1}{1-\delta}\Bigr),\\[-0.5mm]
    E_3: Y^2 &= (X-1)(X-\delta)\Bigl(X-\frac{\delta-1}{\delta}\Bigr)\hspace{-0.3mm}\Bigl(X-\frac{1}{1-\delta}\Bigr),\\[-0.5mm]
    E_4: Y^2 &= X(X-1)\{X-(1-\delta)(\delta-(\delta^2-\delta+1)^{1/2})^2\},\\[-1.4mm]
    E_5: Y^2 &= X(X-1)\{X-(1-\delta)(\delta+(\delta^2-\delta+1)^{1/2})^2\}\\[-4.4mm]
\end{align*}
with\vspace{-0.7mm}
\begin{equation}\label{eq:deltaof10}
    \delta + \frac{\delta-1}{\delta} + \frac{1}{1-\delta} = \frac{A+6}{4}.\vspace{1.5mm}
\end{equation}
Therefore, the curve $H$ is superspecial if and only if the five elliptic curves $E_1,E_2,E_3,E_4,E_5$ are all supersingular.\vspace{-1mm}
\begin{proof}
For the values $(\lambda,\mu,\nu)$ defined in \eqref{eq:lmn}, it follows from the relation \eqref{eq:abc-type10} that\vspace{1mm}
\begin{align*}
    -(\lambda+\mu+\nu) &= \frac{2+a}{2-a} + \frac{2+b}{2-b} + \frac{2+c}{2-c}\\
    &= \frac{24-4(a+b+c)-2(ab+bc+ca)+3abc}{8-4(a+b+c)+2(ab+bc+ca)-abc} = -\frac{A+6}{4},\\
    \lambda\mu+\mu\nu+\nu\lambda &= \frac{2+a}{2-a}\cdot \frac{2+b}{2-b} + \frac{2+b}{2-b}\cdot \frac{2+c}{2-c} + \frac{2+c}{2-c}\cdot \frac{2+a}{2-a}\\
    &= \frac{24+4(a+b+c)-2(ab+bc+ca)-3abc}{8-4(a+b+c)+2(ab+bc+ca)-abc} = \frac{A-6}{4},\\
    -\lambda\mu\nu &= \frac{2+a}{2-a}\cdot \frac{2+b}{2-b} \cdot \frac{2+c}{2-c}\\
    &= \frac{8+4(a+b+c)+2(ab+bc+ca)+abc}{8-4(a+b+c)+2(ab+bc+ca)-abc} = 1,\\[-5.5mm]
\end{align*}
whence\vspace{-0.2mm}
\begin{align*}
    (X-\lambda)(X-\mu)(X-\nu) &= X^3 - (\lambda+\mu+\nu)X^2 +(\lambda\mu + \mu\nu + \nu\lambda)X - \lambda\mu\nu\\[-0.3mm]
    &= X^3 - \frac{A+6}{4}X^2 + \frac{A-6}{4}X + 1\\[-0.5mm]
    &= (X-\delta)\Bigl(X-\frac{\delta-1}{\delta}\Bigr)\hspace{-0.3mm}\Bigl(X-\frac{1}{1-\delta}\Bigr).
\end{align*}
Therefore, applying Theorem~\ref{thm:jacof4-1}, one can see that the Jacobian of $H$ is decomposed into the product of $E_1,E_2,E_3$ and that of the genus-2 curve\vspace{0.2mm}
\begin{equation}\label{eq:g2-type10}
    C: Y^2 = X(X-1)(X-\delta)\Bigl(X-\frac{\delta-1}{\delta}\Bigr)\hspace{-0.3mm}\Bigl(X-\frac{1}{1-\delta}\Bigr).\vspace{1.4mm}
\end{equation}
This curve corresponds to Class~(2) in \cite[p.\,130]{IKO}, and thus its Jacobian is Richelot isogenous to $E_4 \hspace{0.2mm}\times E_5$ as shown in \cite[Remark 1.4]{IKO}.
From the above discussions, we obtain the Jacobian decomposition ${\rm Jac}(H) \sim E_1 \times E_2 \times E_3 \times E_4 \times E_5$, whose degree is a power of 2.
\end{proof}
\end{proposition}\vspace{-2mm}
The $j$-invariants of each elliptic curve in Proposition \ref{prp:jacof10} can be computed as\vspace{-2.2mm}
\begin{align*}
    \begin{split}
    j(E_1) = j(E_2) = j(E_3) &= 256\frac{(\delta^2-\delta+1)^3}{\delta^2(\delta-1)^2} = 16(A^2+108),\\[-0.5mm]
    j(E_4) &= 256\frac{(\epsilon^2-\epsilon+1)^3}{\epsilon^2(\epsilon-1)^2} \,\text{ and }\, j(E_5) = 256\frac{(\epsilon'^2-\epsilon'+1)^3}{\epsilon'^2(\epsilon'-1)^2},
    \end{split}\\[-7.1mm]
\end{align*}
where we define $\epsilon \coloneqq (1-\delta)(\delta-(\delta^2-\delta+1)^{1/2})^2$ and $\epsilon' \coloneqq (1-\delta)(\delta+(\delta^2-\delta+1)^{1/2})^2$.
Since $E_1 \hspace{-0.1mm}\cong\hspace{-0.2mm} E_2 \hspace{-0.1mm}\cong\hspace{-0.2mm} E_3$, it suffices to check that $E_1,E_4$, and $E_5$ are supersingular in order to show that $H$ is superspecial.
Also, we obtain the following statement:\vspace{-4.5mm}
\begin{corollary}
The curve $H$ of Type {\bf 12} is superspecial if and only if the elliptic curve with $j$-invariant $16384/5$ is supersingular.\vspace{-0.7mm}
\begin{proof}
In the case where $A^2=484/5$, a tedious computation shows that\vspace{0.3mm}
\[
    j(E_1) = j(E_2) = j(E_3) = j(E_4) = j(E_5) = 16384/5,\vspace{1.3mm}
\]
which completes the proof.
\end{proof}
\end{corollary}\vspace{-2mm}
At the end of this section, let us show that superspecial curves $H$ of the form \eqref{eq:type10} can be constructed from only arithmetic operations over $\mathbb{F}_{p^2}$.\vspace{-4.7mm}
\begin{proposition}\label{prop:rationalof10}
If the curve $H\hspace{-0.2mm}$ given in \eqref{eq:type10} is superspecial, then $\delta$ and $(\delta^2-\delta+1)^{1/2}$ belong to $\mathbb{F}_{p^2}$.\vspace{-1mm}
\begin{proof}
By the assumption, the genus-2 curve $C$ defined as in \eqref{eq:g2-type10} is also superspecial.
Then, thanks to \cite[Main Theorem A]{Ohashi}, we obtain that $\delta$ and $\delta-\frac{\delta-1}{\delta}$ are both squares in $\mathbb{F}_{p^2}\hspace{-0.3mm}$, which proves the assertion.
\end{proof}
\end{proposition}\vspace{-4mm}

\section{ Algorithms to enumerate superspecial our curves}\label{sec5}
In this section, we propose algorithms to enumerate superspecial hyperelliptic genus-5 curves whose automorphism groups contain $\hspace{-0.2mm}\mathbf{C}_2^3$, exploiting the properties of the curves of each type introduced in the previous sections.

\subsection{ Generic case: ${\rm Aut}(H) \cong \mathbf{C}_2^3$}
As a consequence of the results established in Section~\ref{sec3}, we can construct superspecial hyperelliptic genus-5 curves of the form $H_{a,b,c}$ given in equation \eqref{eq:Habc} from superspecial genus-2 curves in Rosenhain form.
In particular, in the generic case, the isomorphism testing among these curves can be efficiently performed using Theorem \ref{thm:isomof4-1}.
Then, the algorithm is summarized in Algorithm~\ref{alg:type4-1}.\vspace{-0.8mm}
\begin{algorithm}[htbp]
\caption{}\label{alg:type4-1}\vspace{0.8mm}
{\bf Input:} A prime integer $p > 11$.\\
{\bf Output:} A list of isomorphism classes of of superspecial curves $\hspace{-0.2mm}H_{a,b,c}$ of Type {\bf 4-1} in characteristic $p$.\vspace{-1.3mm}
\begin{enumerate}
    \setlength{\itemindent}{8.5mm}
    \item[\emph{Step 1:}] \,List all superspecial genus-2 curves\vspace{-2.2mm}
    \[
        C_{\lambda,\mu,\nu}: y^2 = x(x-1)(x-\lambda)(x-\mu)(x-\nu)\vspace{-0.9mm}
    \]
    in Rosenhain form, and denote the set of such curves by $\textsf{SSp}_2(p)$.
    This can be done by generating the superspecial Richelot isogeny graph (cf. \cite[Section~5A]{KHH20}).\vspace{1mm}
    \item[\emph{Step 2:}] \,For each $C_{\lambda,\mu,\nu} \hspace{-0.2mm}\in \textsf{SSp}_2(p)$, check whether the three elliptic curves $E_1,E_2,E_3$ given in \eqref{eq:ellof4-1} are all supersingular.
    If this is the case, then compute $a,b,c$ by using the relation \eqref{eq:lmn}; specifically, set\vspace{-2.1 mm}
    \[
        a \coloneqq 2\frac{\lambda+1}{\lambda-1}, \quad b \coloneqq 2\frac{\mu+1}{\mu-1}, \,\text{ and }\, c \coloneqq 2\frac{\nu+1}{\nu-1}\vspace{-1.6mm}
    \]
    and store the resulting triple $(a,b,c)$.\vspace{1mm}
    \item[\emph{Step 3:}] \,For each $(a,b,c)$ stored in \emph{Step 2}, check whether none of the permutations of the following six sets is already contained in the list $\hspace{-0.3mm}\mathcal{L}$:\vspace{-2mm}
    \begin{align*}
        \{a,b,c\},\{-a,-b,-c\},\hspace{0.5mm}&\textstyle\{\frac{12-2a}{2+a},\frac{12-2b}{2+b},\frac{12-2c}{2+c}\},\{-\frac{12-2a}{2+a},-\frac{12-2b}{2+b},-\frac{12-2c}{2+c}\},\\[-1.1mm]
        &\textstyle\{\frac{12+2a}{2-a},\frac{12+2b}{2-b},\frac{12+2c}{2-c}\},\{-\frac{12+2a}{2-a},-\frac{12+2b}{2-b},-\frac{12+2c}{2-c}\}.\\[-7.2mm]
    \end{align*}
    If this is the case, then we add $(a,b,c)$ to the list $\mathcal{L}$ as long as ${\rm Aut}(H_{a,b,c}) \cong \mathbf{C}_2^3$.\vspace{1mm}
    \item[\emph{Step 4:}] \,Output the list of our curves $H_{a,b,c}$ for $(a,b,c) \in \mathcal{L}$.\vspace{-1.2mm}
\end{enumerate}
\end{algorithm}\vspace{-1.2mm}

All computations occupied in Algorithm~\ref{alg:type4-1} can be done over $\mathbb{F}_{p^2}$\hspace{-0.3mm}, and its complexity is $\widetilde{O}(p^3)$.
Indeed, \hspace{0.2mm}\emph{Step~1} enumerates superspecial genus-2 curves with $\widetilde{O}(p^3)$ arithmetic operations over $\mathbb{F}_{p^2}$\hspace{-0.3mm}, and it is also known (cf. \cite[Section~5.1]{OKH23}) that $\hspace{-0.2mm}\#\textsf{SSp}_2(p)=O(p^3)$.
All subsequent steps, such as testing supersingularity of elliptic curves and computing automorphism groups of hyperelliptic curves, run in polynomial time in $\log(p)$.
Hence, this completes the complexity analysis of Algorithm~\ref{alg:type4-1}.

\subsection{ Special case: ${\rm Aut}(H) \supset \mathbf{C}_2^2 \rtimes \mathbf{C}_4$}
In this subsection, we propose an algorithm for enumerating superspecial hyperelliptic genus-5 curves given in equation \eqref{eq:type7}, whose automorphism groups contain $\mathbf{C}_2^2 \rtimes\hspace{0.2mm} \mathbf{C}_4$.
We recall from Corollary~\ref{cor:modof7} that it suffices to consider the case where $p \equiv 3 \pmod{4}$.
Then, the algorithm is summarized in Algorithm~\ref{alg:type7}.\vspace{-0.8mm}

\begin{algorithm}[htbp]
\caption{}\label{alg:type7}\vspace{0.8mm}
{\bf Input:} A prime integer $p > 11$ with $p \equiv 3 \pmod{4}$.\\
{\bf Output:} A list of isomorphism classes of of superspecial curves whose automorphism groups contain $\mathbf{C}_2^2 \hspace{-0.2mm}\rtimes\hspace{-0.2mm} {\bf C}_4$ in characteristic $p$.\vspace{-1.3mm}
\begin{enumerate}
    \setlength{\itemindent}{8.5mm}
    \item[\emph{Step 1:}] \,Compute the complete list $\mathcal{S}_p$ of supersingular $j$-invariants in characteristic $p$.\\
    This can be done by generating the supersingular 2-isogeny graph.\vspace{1mm}
    \item[\emph{Step 2:}] \,For each $j \in \mathcal{S}_p$, compute all $\delta \in \mathbb{F}_{p^2}\hspace{-0.2mm}$ with\vspace{-2.8mm}
    \[
        256\frac{(\delta^2-\delta+1)^3}{\delta^2(\delta-1)^2} = j, \quad \delta \notin \{0,\pm 1\},\vspace{-2.6mm}
    \]
    and check whether\vspace{0.2mm}
    \begin{equation}\label{eq:condof7}
        64\frac{(2\delta-1)^3(2\delta+1)^3}{\delta^2} \in \mathcal{S}_p.\vspace{1mm}
    \end{equation}
    If this is the case, then compute $A$ by using relation \eqref{eq:deltaof7}; \,specifically, set\vspace{-2.5mm}
    \[
        A \coloneqq 2+\frac{16\delta}{(\delta-1)^2}\vspace{-1.8mm}
    \]
    and store the resulting value $A$.\vspace{1mm}
    \item[\emph{Step 3:}] \,For each $A$ stored in \emph{Step 2}, check whether the hyperelliptic genus-5 curve $H$ given in \eqref{eq:type7} is not isomorphic to any element of the list $\mathcal{L}$.
    If this is the case, then\\ we add the curve $H$ to the list $\mathcal{L}$.\vspace{1mm}
    \item[\emph{Step 4:}] \,Output the list $\mathcal{L}$.\vspace{-1.2mm}
\end{enumerate}
\end{algorithm}\vspace{-1.2mm}

Due to Proposition~\ref{prop:rationalof7}, all computations required for Algorithm~\ref{alg:type7} are performed over $\mathbb{F}_{p^2}$.
The complexity of \emph{Step 1} is $\widetilde{O}(p)$, and that of \emph{Step 2} is also $\widetilde{O}(p)$.
Since the number of supersingular $j$-invariants $\in \mathcal{S}_p$ is known to be $O(p)$, the total number of $\delta$ computed in \emph{Step 2} is also $O(p)$.
Also, the probability that condition \eqref{eq:condof7} is satisfied for each $\delta$ is approximately $1/p$, we can estimate as follows:\medskip\\
\noindent {\bf Heuristic.} \,The number of $A$'s listed in \emph{Step 2} is $O(1)$.\medskip\\
\noindent Assuming this heuristic, we estimate the complexity of \emph{Step 3} as $\widetilde{O}(1)$.
Therefore, the total complexity of Algorithm~\ref{alg:type7} is expected to be $\tilde{O}(p)$.

\subsection{ Special case: ${\rm Aut}(H) \supset \mathbf{C}_2 \times \mathbf{D}_{12}$}
In this subsection, we propose an algorithm for enumerating superspecial hyperelliptic genus-5 curves given in equation \eqref{eq:type9} whose automorphism groups contain $\mathbf{C}_2 \hspace{0.2mm}\times\hspace{0.2mm} \mathbf{D}_{12}$.
The algorithm is summarized in Algorithm~\ref{alg:type9}.

\begin{algorithm}[htbp]
\caption{}\label{alg:type9}\vspace{0.8mm}
{\bf Input:} A prime integer $p > 11$.\\
{\bf Output:} A list of isomorphism classes of of superspecial curves whose automorphism groups contain $\mathbf{C}_2 \hspace{-0.2mm}\times\hspace{-0.2mm} {\bf D}_{12}$ in characteristic $p$.\vspace{-1.3mm}
\begin{enumerate}
    \setlength{\itemindent}{8.5mm}
    \item[\emph{Step 1:}] \,Compute $F \coloneqq {\rm gcd}(F_1,F_2,F_3)$, where\vspace{-2.4mm}
    \begin{align*}
        F_1 &\coloneqq \left\{
        \begin{array}{ll}
            G^{((p-1)/6)}(1/2,1/6,2/3 \mid z) & \text{ if }\, p \equiv 1 \!\!\!\pmod{6},\\
            G^{((p-5)/6)}(1/2,5/6,4/3 \mid z) & \text{ if }\, p \equiv 5 \!\!\!\pmod{6},
        \end{array}
        \right.\\[-0.3mm]
        F_2 &\coloneqq \left\{
        \begin{array}{ll}
            G^{((p-1)/3)}(1/2,1/3,5/6 \mid z) & \text{ if }\, p \equiv 1 \!\!\!\pmod{6},\\
            G^{((p-2)/3)}(1/2,2/3,7/6 \mid z) & \text{ if }\, p \equiv 5 \!\!\!\pmod{6},
        \end{array}
        \right.\\[-0.3mm]
        F_3 &\coloneqq G^{((p-1)/2)}(1/2,1/2,1 \mid z),\\[-7.3mm]
    \end{align*}
    and store all roots $\delta \in \mathbb{F}_{p^2}\hspace{-0.3mm}$ of the polynomial $F$.\vspace{1mm}
    \item[\emph{Step 2:}] \,For each $\delta$ stored in \emph{Step 1}, check whether the hyperelliptic genus-5 curve $H$ given in \eqref{eq:deltaof9} is not isomorphic to any element of the list $\mathcal{L}$.
    If this is the case, we add the curve $H$ to the list $\mathcal{L}$.\vspace{1mm}
    \item[\emph{Step 3:}] \,Output the list $\mathcal{L}$.\vspace{-1.2mm}
\end{enumerate}
\end{algorithm}\vspace{-1.2mm}

Due to Proposition~\ref{prop:rationalof9}, all computations required for Algorithm~\ref{alg:type9} are performed over $\mathbb{F}_{p^2}$.
The degree of the polynomial $F$ computed in \emph{Step 1} is theoretically bounded above by that of $F_1$, and experiments for small $p$ indicate the following estimate:\medskip\\
\noindent {\bf Heuristic.} \,The degree of $F$ computed in \emph{Step 1} is $O(1)$.\medskip\\
\noindent Assuming this heuristic, we estimate that the complexity of \emph{Step 1} is $\widetilde{O}(p)$ and that of \emph{Step 2} is $\widetilde{O}(1)$.
Therefore, the total complexity of Algorithm~\ref{alg:type9} is expected to be $\tilde{O}(p)$.

\subsection{ Special case: ${\rm Aut}(H) \supset \mathbf{C}_2 \times \mathbf{A}_4$}
In this subsection, we propose an algorithm for enumerating superspecial hyperelliptic genus-5 curves given in equation \eqref{eq:type10} whose automorphism groups contain $\mathbf{C}_2 \hspace{0.2mm}\times\hspace{0.2mm} \mathbf{A}_4$.
The algorithm is described in Algorithm~\ref{alg:type10}.

\begin{algorithm}[htbp]
\caption{}\label{alg:type10}\vspace{0.8mm}
{\bf Input:} A prime integer $p > 11$.\\
{\bf Output:} A list of isomorphism classes of of superspecial curves whose automorphism groups contain $\mathbf{C}_2 \hspace{-0.2mm}\times\hspace{-0.2mm} {\bf A}_4$ in characteristic $p$.\vspace{-1.3mm}
\begin{enumerate}
    \setlength{\itemindent}{8.5mm}
    \item[\emph{Step 1:}] \,Compute the complete list $\mathcal{S}_p$ of supersingular $j$-invariants in characteristic $p$.\\
    This can be done by generating the supersingular 2-isogeny graph.\vspace{1mm}
    \item[\emph{Step 2:}] \,For each $j \in \mathcal{S}_p$, compute all $\delta \in \mathbb{F}_{p^2}\hspace{-0.2mm}$ with\vspace{-2.4mm}
    \[
        256\frac{(\delta^2-\delta+1)^3}{\delta^2(\delta-1)^2} = j, \quad \delta \notin\hspace{-0.2mm} \{0,1\}, \quad \delta^2-\delta+1 \neq 0\vspace{-2.6mm}
    \]
    and check whether $\delta^2-\delta+1$ is square in $\mathbb{F}_{p^2}\hspace{-0.3mm}$, as well as\vspace{-2.5mm}
    \[
        256\frac{(\epsilon^2-\epsilon+1)^3}{\epsilon^2(\epsilon-1)^2},\,256\frac{(\epsilon'^2-\epsilon'+1)^3}{\epsilon'^2(\epsilon'-1)^2} \in \mathcal{S}_p.\vspace{-1.5mm}
    \]
    with $\epsilon \coloneqq (1-\delta)(\delta-(\delta^2-\delta+1)^{1/2})^2,\hspace{1mm}\epsilon' \coloneqq (1-\delta)(\delta+(\delta^2-\delta+1)^{1/2})^2$.
    If this is the case, then compute $A$ by using relation \eqref{eq:deltaof10}; \,specifically, set\vspace{-2.2mm}
    \[
        A \coloneqq 2\frac{(\delta+1)(\delta-2)(2\delta-1)}{\delta(\delta-1)}\vspace{-1.7mm}
    \]
    and store the resulting value $A$.\vspace{1mm}
    \item[\emph{Step 3:}] \,For each $A$ stored in \emph{Step 2}, check whether the hyperelliptic genus-5 curve $H$ given in \eqref{eq:type10} is not isomorphic to any element of the list $\mathcal{L}$.
    If this is the case, then\\ we add the curve $H$ to the list $\mathcal{L}$.\vspace{1mm}
    \item[\emph{Step 4:}] \,Output the list $\mathcal{L}$.\vspace{-1.2mm}
\end{enumerate}
\end{algorithm}

Due to Proposition~\ref{prop:rationalof10}, all computations required for Algorithm~\ref{alg:type10} are performed over $\mathbb{F}_{p^2}$.
The complexity of \emph{Step 1} is $\widetilde{O}(p)$, and that of \emph{Step 2} is also $\widetilde{O}(p)$.
Moreover, as in Section 5.2, we can estimate as follows:\medskip\\
\noindent {\bf Heuristic.} \,The number of $A$'s listed in \emph{Step 2} is $O(1)$.\medskip\\
\noindent Assuming this heuristic, we estimate the complexity of \emph{Step 3} as $\widetilde{O}(1)$.
Therefore, the total complexity of Algorithm~\ref{alg:type10} is expected to be $\tilde{O}(p)$.

\section{ Experimental results}\label{sec6}
We implemented in \textsf{Magma} Computational Algebra System the algorithms introduced in Section \ref{sec5}.
By executing them for all primes $p$ with $11 < p < 1000$, we have obtained Theorem \ref{thm:main} as our main result.
For the number of isomorphism classes of superspecial curves enumerated for each prime $p$, see Tables \ref{tbl:13-450} and \ref{tbl:450-1000}.

\begin{table}[htbp]
    \centering
    \begin{minipage}{0.47\hsize}
    \hspace{30mm}Types\\[-1mm]
    \begin{tabular}{c||c|ccc|ccc||c}
        $p$ & {\bf 4-1} & \hspace{-0.5mm}{\bf 7}\hspace{-0.5mm} & \hspace{-0.5mm}{\bf 9}\hspace{-0.5mm} & \hspace{-1mm}{\bf 10}\hspace{-1mm} & \hspace{-1mm}{\bf 11}\hspace{-1mm} & \hspace{-1mm}{\bf 12}\hspace{-1mm} & \hspace{-1mm}{\bf 15}\hspace{-1mm} & All\\\hline
        $13$ & 0 & 0 & 0 & 0 & 0 & 0 & 0 & 0\\
        $17$ & 0 & 0 & 0 & 0 & 0 & 0 & 0 & 0\\
        $19$ & 0 & 0 & 0 & 0 & 0 & 0 & 0 & 0\\
        $23$ & 0 & 0 & 0 & 0 & 1 & 0 & 1 & 2\\
        $29$ & 0 & 0 & 0 & 1 & 0 & 0 & 0 & 1\\
        $31$ & 1 & 1 & 0 & 1 & 0 & 0 & 0 & 3\\
        $37$ & 0 & 0 & 0 & 0 & 0 & 0 & 0 & 0\\
        $41$ & 0 & 0 & 0 & 1 & 0 & 0 & 0 & 0\\
        $43$ & 0 & 0 & 0 & 1 & 0 & 0 & 0 & 0\\
        $47$ & 1 & 0 & 0 & 0 & 1 & 0 & 1 & 3\\
        $53$ & 0 & 0 & 0 & 0 & 0 & 0 & 0 & 0\\
        $59$ & 0 & 1 & 0 & 1 & 1 & 0 & 1 & 4\\
        $61$ & 0 & 0 & 0 & 0 & 0 & 0 & 0 & 0\\
        $67$ & 0 & 0 & 0 & 0 & 0 & 0 & 0 & 0\\
        $71$ & 4 & 1 & 0 & 2 & 1 & 0 & 1 & 9\\
        $73$ & 0 & 0 & 0 & 0 & 0 & 0 & 0 & 0\\
        $79$ & 1 & 1 & 0 & 0 & 0 & 0 & 0 & 2\\
        $83$ & 2 & 1 & 0 & 1 & 1 & 0 & 1 & 6\\
        $89$ & 2 & 0 & 0 & 2 & 0 & 0 & 0 & 4\\
        $97$ & 0 & 0 & 0 & 0 & 0 & 0 & 0 & 0\\
        $101$ & 2 & 0 & 0 & 2 & 0 & 0 & 0 & 4\\
        $103$ & 0 & 1 & 0 & 0 & 0 & 0 & 0 & 1\\
        $107$ & 3 & 1 & 0 & 1 & 1 & 0 & 1 & 7\\
        $109$ & 1 & 0 & 0 & 1 & 0 & 0 & 0 & 2\\
        $113$ & 1 & 0 & 0 & 1 & 0 & 0 & 0 & 2\\
        $127$ & 0 & 2 & 0 & 0 & 0 & 0 & 0 & 2\\
        $131$ & 10 & 0 & 1 & 1 & 1 & 1 & 1 & 15\\
        $137$ & 0 & 0 & 0 & 1 & 0 & 0 & 0 & 1\\
        $139$ & 1 & 2 & 0 & 0 & 0 & 0 & 0 & 3\\
        $149$ & 1 & 0 & 0 & 1 & 0 & 0 & 0 & 2\\
        $151$ & 5 & 3 & 0 & 0 & 0 & 0 & 0 & 8\\
        $157$ & 0 & 0 & 0 & 2 & 0 & 0 & 0 & 2\\
        $163$ & 4 & 0 & 0 & 0 & 0 & 0 & 0 & 4\\
        $167$ & 8 & 0 & 1 & 1 & 1 & 0 & 1 & 12\\
        $173$ & 2 & 0 & 0 & 2 & 0 & 0 & 0 & 4\\
        $179$ & 3 & 2 & 0 & 1 & 1 & 0 & 1 & 8\\
        $181$ & 0 & 0 & 0 & 0 & 0 & 0 & 0 & 0\\
        $191$ & 23 & 2 & 0 & 4 & 1 & 0 & 1 & 31\\
        $193$ & 0 & 0 & 0 & 0 & 0 & 0 & 0 & 0\\
        $197$ & 0 & 0 & 0 & 0 & 0 & 0 & 0 & 0\\
        $199$ & 4 & 3 & 0 & 0 & 0 & 0 & 0 & 7
    \end{tabular}
    \end{minipage}\hspace{2mm}
    \begin{minipage}{0.47\hsize}
    \hspace{30mm}Types\\[-1mm]
    \begin{tabular}{c||c|ccc|ccc||c}
        $p$ & {\bf 4-1} & \hspace{-0.5mm}{\bf 7}\hspace{-0.5mm} & \hspace{-0.5mm}{\bf 9}\hspace{-0.5mm} & \hspace{-1mm}{\bf 10}\hspace{-1mm} & \hspace{-1mm}{\bf 11}\hspace{-1mm} & \hspace{-1mm}{\bf 12}\hspace{-1mm} & \hspace{-1mm}{\bf 15}\hspace{-1mm} & All\\\hline
        $211$ & 2 & 0 & 0 & 0 & 0 & 0 & 0 & 2\\
        $223$ & 7 & 2 & 0 & 1 & 0 & 0 & 0 & 10\\
        $227$ & 2 & 0 & 0 & 2 & 1 & 0 & 1 & 6\\
        $229$ & 0 & 0 & 1 & 1 & 0 & 0 & 0 & 2\\
        $233$ & 0 & 0 & 0 & 0 & 0 & 0 & 0 & 0\\
        $239$ & 26 & 3 & 0 & 3 & 1 & 0 & 1 & 34\\
        $241$ & 0 & 0 & 0 & 0 & 0 & 0 & 0 & 0\\
        $251$ & 6 & 4 & 1 & 2 & 1 & 1 & 1 & 16\\
        $257$ & 14 & 0 & 0 & 2 & 0 & 0 & 0 & 16\\
        $263$ & 14 & 1 & 0 & 2 & 1 & 0 & 1 & 19\\
        $269$ & 2 & 0 & 0 & 4 & 0 & 0 & 0 & 6\\
        $271$ & 4 & 1 & 0 & 0 & 0 & 0 & 0 & 5\\
        $277$ & 0 & 0 & 0 & 0 & 0 & 0 & 0 & 0\\
        $281$ & 6 & 0 & 0 & 4 & 0 & 0 & 0 & 10\\
        $283$ & 0 & 5 & 0 & 2 & 0 & 0 & 0 & 7\\
        $293$ & 5 & 0 & 0 & 1 & 0 & 0 & 0 & 6\\
        $307$ & 2 & 0 & 0 & 0 & 0 & 0 & 0 & 2\\
        $311$ & 34 & 4 & 0 & 1 & 1 & 0 & 1 & 41\\
        $313$ & 0 & 0 & 0 & 0 & 0 & 0 & 0 & 0\\
        $317$ & 2 & 0 & 0 & 0 & 0 & 0 & 0 & 2\\
        $331$ & 0 & 0 & 0 & 0 & 0 & 0 & 0 & 0\\
        $337$ & 2 & 0 & 0 & 0 & 0 & 0 & 0 & 2\\
        $347$ & 2 & 3 & 0 & 1 & 1 & 0 & 1 & 8\\
        $349$ & 5 & 0 & 0 & 0 & 0 & 0 & 0 & 5\\
        $353$ & 1 & 0 & 0 & 0 & 0 & 0 & 0 & 1\\
        $359$ & 33 & 7 & 0 & 1 & 1 & 0 & 1 & 43\\
        $367$ & 2 & 1 & 0 & 0 & 0 & 0 & 0 & 3\\
        $373$ & 0 & 0 & 0 & 0 & 0 & 0 & 0 & 0\\
        $379$ & 2 & 0 & 0 & 0 & 0 & 0 & 0 & 2\\
        $383$ & 38 & 7 & 1 & 2 & 1 & 0 & 1 & 50\\
        $389$ & 1 & 0 & 0 & 4 & 0 & 0 & 0 & 5\\
        $397$ & 0 & 0 & 0 & 1 & 0 & 0 & 0 & 1\\
        $401$ & 6 & 0 & 0 & 3 & 0 & 0 & 0 & 9\\
        $409$ & 2 & 0 & 0 & 0 & 0 & 0 & 0 & 2\\
        $419$ & 11 & 3 & 0 & 1 & 1 & 0 & 1 & 17\\
        $421$ & 0 & 0 & 0 & 0 & 0 & 0 & 0 & 0\\
        $431$ & 20 & 3 & 1 & 2 & 1 & 0 & 1 & 28\\
        $433$ & 0 & 0 & 0 & 1 & 0 & 0 & 0 & 1\\
        $439$ & 9 & 2 & 0 & 1 & 0 & 0 & 0 & 12\\
        $443$ & 3 & 0 & 0 & 2 & 1 & 0 & 1 & 7\\
        $449$ & 2 & 0 & 0 & 1 & 0 & 0 & 0 & 3
    \end{tabular}
    \end{minipage}\vspace{2.5mm}

    \begin{minipage}{0.83\hsize}
    \caption{\hspace{0.5mm}The number of isomorphism classes of superspecial genus-5 hyperelliptic curves $H$ such that ${\rm Aut}(H) \supset (\mathbb{Z}/2\mathbb{Z})^3$ in characteristic $13 \leq p < 450$}\label{tbl:13-450}
    \end{minipage}
\end{table}

\begin{table}[htbp]
    \centering
    \begin{minipage}{0.47\hsize}
    \hspace{30mm}Types\\[-1mm]
    \begin{tabular}{c||c|ccc|ccc||c}
        $p$ & {\bf 4-1} & \hspace{-0.5mm}{\bf 7}\hspace{-0.5mm} & \hspace{-0.5mm}{\bf 9}\hspace{-0.5mm} & \hspace{-1mm}{\bf 10}\hspace{-1mm} & \hspace{-1mm}{\bf 11}\hspace{-1mm} & \hspace{-1mm}{\bf 12}\hspace{-1mm} & \hspace{-1mm}{\bf 15}\hspace{-1mm} & All\\\hline
        $457$ & 0 & 0 & 0 & 2 & 0 & 0 & 0 & 2\\
        $461$ & 3 & 0 & 0 & 2 & 0 & 0 & 0 & 5\\
        $463$ & 6 & 2 & 0 & 0 & 0 & 0 & 0 & 8\\
        $467$ & 3 & 0 & 0 & 0 & 1 & 0 & 1 & 5\\
        $479$ & 30 & 4 & 1 & 4 & 1 & 0 & 1 & 41\\
        $487$ & 0 & 2 & 0 & 0 & 0 & 0 & 0 & 2\\
        $491$ & 9 & 3 & 1 & 5 & 1 & 1 & 1 & 21\\
        $499$ & 0 & 0 & 0 & 1 & 0 & 0 & 0 & 1\\
        $503$ & 16 & 7 & 1 & 2 & 1 & 0 & 1 & 28\\
        $509$ & 2 & 0 & 0 & 1 & 0 & 0 & 0 & 3\\
        $521$ & 4 & 0 & 0 & 3 & 0 & 0 & 0 & 7\\
        $523$ & 2 & 0 & 0 & 0 & 0 & 0 & 0 & 2\\
        $541$ & 0 & 0 & 0 & 0 & 0 & 0 & 0 & 0\\
        $547$ & 1 & 2 & 0 & 0 & 0 & 0 & 0 & 3\\
        $557$ & 1 & 0 & 0 & 1 & 0 & 0 & 0 & 2\\
        $563$ & 9 & 0 & 0 & 2 & 1 & 0 & 1 & 13\\
        $569$ & 2 & 0 & 0 & 2 & 0 & 0 & 0 & 4\\
        $571$ & 0 & 0 & 0 & 0 & 0 & 0 & 0 & 0\\
        $577$ & 1 & 0 & 0 & 0 & 0 & 0 & 0 & 1\\
        $587$ & 0 & 2 & 0 & 2 & 1 & 0 & 1 & 6\\
        $593$ & 1 & 0 & 0 & 2 & 0 & 0 & 0 & 3\\
        $599$ & 25 & 5 & 1 & 3 & 1 & 1 & 1 & 37\\
        $601$ & 0 & 0 & 0 & 4 & 0 & 0 & 0 & 4\\
        $607$ & 7 & 2 & 0 & 1 & 0 & 0 & 0 & 9\\
        $613$ & 0 & 0 & 0 & 0 & 0 & 0 & 0 & 0\\
        $617$ & 3 & 0 & 0 & 1 & 0 & 0 & 0 & 4\\
        $619$ & 0 & 1 & 0 & 0 & 0 & 0 & 0 & 1\\
        $631$ & 5 & 1 & 0 & 0 & 0 & 0 & 0 & 6\\
        $641$ & 2 & 0 & 0 & 3 & 0 & 0 & 0 & 5\\
        $643$ & 0 & 0 & 0 & 0 & 0 & 0 & 0 & 0\\
        $647$ & 19 & 4 & 0 & 0 & 1 & 0 & 1 & 25\\
        $653$ & 1 & 0 & 0 & 2 & 0 & 0 & 0 & 3\\
        $659$ & 13 & 5 & 0 & 4 & 1 & 0 & 1 & 24\\
        $661$ & 1 & 0 & 0 & 0 & 0 & 0 & 0 & 1\\
        $673$ & 0 & 0 & 0 & 0 & 0 & 0 & 0 & 0\\
        $677$ & 0 & 0 & 0 & 2 & 0 & 0 & 0 & 2\\
        $683$ & 0 & 0 & 0 & 0 & 1 & 0 & 1 & 2\\
        $691$ & 3 & 2 & 0 & 2 & 0 & 0 & 0 & 7\\
        $701$ & 1 & 0 & 0 & 3 & 0 & 0 & 0 & 4\\
        $709$ & 0 & 0 & 0 & 0 & 0 & 0 & 0 & 0\\
        $719$ & 53 & 7 & 2 & 6 & 1 & 0 & 1 & 70
    \end{tabular}
    \end{minipage}\hspace{2mm}
    \begin{minipage}{0.47\hsize}
    \hspace{30mm}Types\\[-1mm]
    \begin{tabular}{c||c|ccc|ccc||c}
        $p$ & {\bf 4-1} & \hspace{-0.5mm}{\bf 7}\hspace{-0.5mm} & \hspace{-0.5mm}{\bf 9}\hspace{-0.5mm} & \hspace{-1mm}{\bf 10}\hspace{-1mm} & \hspace{-1mm}{\bf 11}\hspace{-1mm} & \hspace{-1mm}{\bf 12}\hspace{-1mm} & \hspace{-1mm}{\bf 15}\hspace{-1mm} & All\\\hline
        $727$ & 4 & 4 & 0 & 2 & 0 & 0 & 0 & 10\\
        $733$ & 0 & 0 & 0 & 2 & 0 & 0 & 0 & 2\\
        $739$ & 4 & 1 & 0 & 1 & 0 & 0 & 0 & 6\\
        $743$ & 15 & 2 & 0 & 3 & 1 & 0 & 1 & 22\\
        $751$ & 9 & 2 & 0 & 0 & 0 & 0 & 0 & 11\\
        $757$ & 1 & 0 & 0 & 0 & 0 & 0 & 0 & 1\\
        $761$ & 1 & 0 & 0 & 1 & 0 & 0 & 0 & 2\\
        $769$ & 0 & 0 & 0 & 0 & 0 & 0 & 0 & 0\\
        $773$ & 5 & 0 & 0 & 0 & 0 & 0 & 0 & 5\\
        $787$ & 0 & 1 & 0 & 0 & 0 & 0 & 0 & 1\\
        $797$ & 0 & 0 & 0 & 2 & 0 & 0 & 0 & 2\\
        $809$ & 1 & 0 & 0 & 1 & 0 & 0 & 0 & 2\\
        $811$ & 0 & 0 & 0 & 4 & 0 & 0 & 0 & 4\\
        $821$ & 4 & 0 & 0 & 4 & 0 & 0 & 0 & 8\\
        $823$ & 1 & 2 & 0 & 0 & 0 & 0 & 0 & 3\\
        $827$ & 2 & 3 & 0 & 0 & 1 & 0 & 1 & 7\\
        $829$ & 1 & 0 & 0 & 0 & 0 & 0 & 0 & 1\\
        $839$ & 50 & 1 & 2 & 7 & 1 & 0 & 1 & 62\\
        $853$ & 2 & 0 & 0 & 0 & 0 & 0 & 0 & 2\\
        $857$ & 1 & 0 & 0 & 1 & 0 & 0 & 0 & 2\\
        $859$ & 2 & 6 & 0 & 0 & 0 & 0 & 0 & 8\\
        $863$ & 14 & 1 & 0 & 2 & 1 & 0 & 1 & 19\\
        $877$ & 0 & 0 & 0 & 0 & 0 & 0 & 0 & 0\\
        $881$ & 5 & 0 & 0 & 3 & 0 & 0 & 0 & 8\\
        $883$ & 2 & 0 & 0 & 0 & 0 & 0 & 0 & 2\\
        $887$ & 29 & 5 & 1 & 2 & 1 & 0 & 1 & 39\\
        $907$ & 2 & 1 & 0 & 0 & 0 & 0 & 0 & 3\\
        $911$ & 41 & 6 & 0 & 2 & 1 & 0 & 1 & 51\\
        $919$ & 7 & 1 & 0 & 0 & 0 & 0 & 0 & 8\\
        $929$ & 5 & 0 & 0 & 1 & 0 & 0 & 0 & 6\\
        $937$ & 1 & 0 & 0 & 0 & 0 & 0 & 0 & 1\\
        $941$ & 3 & 0 & 0 & 4 & 0 & 0 & 0 & 7\\
        $947$ & 4 & 1 & 0 & 0 & 1 & 0 & 1 & 7\\
        $953$ & 2 & 0 & 0 & 1 & 0 & 0 & 0 & 3\\
        $967$ & 3 & 5 & 0 & 0 & 0 & 0 & 0 & 8\\
        $971$ & 10 & 4 & 0 & 5 & 1 & 0 & 1 & 21\\
        $977$ & 0 & 0 & 0 & 0 & 0 & 0 & 0 & 0\\
        $983$ & 13 & 4 & 0 & 2 & 1 & 0 & 1 & 21\\
        $991$ & 3 & 0 & 0 & 0 & 0 & 0 & 0 & 3\\
        $997$ & 0 & 1 & 0 & 1 & 0 & 0 & 0 & 2\\
        {} & & & & & & & &
        \end{tabular}
    \end{minipage}\vspace{2.5mm}

    \begin{minipage}{0.83\hsize}
    \caption{\hspace{0.5mm}The number of isomorphism classes of superspecial genus-5 hyperelliptic curves $H$ such that ${\rm Aut}(H) \supset (\mathbb{Z}/2\mathbb{Z})^3$ in characteristic $450 < p < 1000$}\label{tbl:450-1000}
    \end{minipage}
\end{table}

\section{ Concluding remarks}\label{sec7}
We presented an algorithm with complexity $\widetilde{O}(p^3)$ arithmetic operations over $\mathbb{F}_{p^2}\hspace{-0.4mm}$ for enumerating superspecial hyperelliptic curves of genus 5 whose automorphism groups contain $(\mathbb{Z}/2\mathbb{Z})^3$.
Executing our algorithm in \textsf{Magma} Computational Algebra System, we succeeded in enumerating such curves for primes $p$ with $11 < p < 1000$.

Our computational results also reveal an interesting phenomenon: \hspace{0.5mm}there are many primes $p$ for which no superspecial hyperelliptic curve of genus $5$ whose automorphism groups contain $(\mathbb{Z}/2\mathbb{Z})^3$ exists in characteristic $p$, in contrast to the case of hyperelliptic curves of genus $4$ whose automorphism groups contain $\hspace{-0.2mm}\mathbf{D}_8$.
Hence, in order to confirm the existence of superspecial hyperelliptic curves of genus $5$ for every prime $p$, it may be necessary to investigate a broader class (for example, Type {\bf 2-1} in Table~\ref{table:classification}).
However, this would require preparing several superspecial curves of genus $3$, which significantly increases the computational complexity (cf. \cite{OOKYN}).
As future work, we plan to construct an efficient enumeration algorithm for the case of non-hyperelliptic curves of genus $5$, analogous to the one developed in this work.

\bibliography{sn-bibliography}
\end{document}